\documentclass[11pt]{article}
\usepackage[UKenglish]{babel}
\usepackage[utf8]{inputenc}
\usepackage[T1]{fontenc}
\usepackage{enumerate}
\usepackage{graphicx}
\usepackage[dvipsnames,svgnames,x11names]{xcolor}
\usepackage{amsmath,amssymb,amsthm}
\usepackage{marginnote}
\usepackage{authblk} 
\usepackage[colorlinks=true, allcolors=blue]{hyperref}
\usepackage[lmargin=25mm,rmargin=25mm,top=25mm,bottom=30mm,paperwidth=210mm,paperheight=11in]{geometry} 

\numberwithin{equation}{section}

\newcommand{\E}{\mathbb{E}}
\newcommand{\pr}{\mathbb{P}}
\newcommand{\PP}{\mathbb{P}}

\newcommand\cA{\mathcal{A}}
\newcommand\cF{\mathcal{F}}
\newcommand\cM{\mathcal{M}}
\newcommand\cN{\mathcal{N}}
\newcommand\cL{\mathcal{L}}
\newcommand{\mbf}[1]{\mathbf{#1}}

\def\-{\text{-}} 

\theoremstyle{plain}
\newtheorem{theorem}{Theorem}[section]
\newtheorem{lemma}[theorem]{Lemma}
\newtheorem{corollary}[theorem]{Corollary}

\theoremstyle{definition}
\newtheorem{remark}[theorem]{Remark}

\def\R{\mathbb{R}}
\def\N{\mathbb{N}}

\newcommand{\norm}[1]{\left\lVert#1\right\rVert}
\newcommand{\ind}[1]{{\bf 1}{\left( #1 \right)}} 

\newcommand{\law}{\operatorname{Law}}
\newcommand{\weq}{\ = \ }
\newcommand{\wapprox}{\ \approx \ }
\newcommand{\wle}{\ \le \ }
\newcommand{\wge}{\ \ge \ }

\newcommand{\abs}[1]{{\lvert#1\rvert}}

\newcommand{\ceil}[1]{\left\lceil #1 \right\rceil}

\begin{document}

\title{Persistence in a large network of locally interacting neurons}

\author{
	Maximiliano Altamirano\footnote{
		Instituto de C\'alculo, Universidad de Buenos Aires/CONICET, Argentina, E-mail: \texttt{maximiliano.altamirano@ic.fcen.uba.ar}.
	},
	Roberto Cortez\footnote{
		Universidad Andr\'es Bello, Departamento de Matem\'aticas. E-mail: \texttt{roberto.cortez.m@unab.cl}. Supported by Iniciaci\'on Fondecyt Grant 11181082 and by Programa Iniciativa Científica Milenio through Nucleus Millenium Stochastic Models of Complex and Disordered Systems.
	},
	Matthieu Jonckheere\footnote{
		Instituto de C\'alculo, Universidad de Buenos Aires/CONICET, Argentina, E-mail: \texttt{mjonckhe@dm.uba.ar}.
	},
	Lasse Leskel\"a\footnote{
		Aalto University,
		School of Science,
		Department of Mathematics and Systems Analysis.
		Otakaari 1, 02015 Espoo, Finland.
		E-mail: 
		\texttt{lasse.leskela@aalto.fi}.
	}
}
\date{}
\maketitle

\begin{abstract}
	This article presents a biological neural network model driven by inhomogeneous Poisson processes accounting for the intrinsic randomness of synapses. The main novelty is the introduction of local interactions: each firing neuron triggers an instantaneous increase in electric potential to a fixed number of randomly chosen neurons.
	We prove that, as the number of neurons approaches infinity, the finite network converges to a nonlinear mean-field process characterised by a jump-type stochastic differential equation. We show that this process displays a phase transition: the activity of a typical neuron in the infinite network either rapidly dies out, or persists forever, depending on the global parameters describing the intensity of interconnection. This provides a way to understand the emergence of persistent activity triggered by weak input signals in large neural networks.
\end{abstract}

\bigskip

\noindent
\textbf{Keywords:} biological neural network,
metastability,
mean-field limit,
jump process,
interacting particle system,
phase transition, propagation of chaos,
nonlinear Markov process
\tableofcontents

\section{Introduction}

\subsection{Biological neural network models}

Since the seminal work by Lapicque \cite{lapicque1907recherches}, the body of mathematical literature on biological neural networks has become extensive. Nevertheless, several fundamental questions still remain open due to the the extreme complexity inherent to this biological system. Among these, an important one concerns the understanding of the phenomenon of \textit{persistence}: given a mild input, the network activity is sustained for surprisingly large times. It has been observed in a diverse set of brain regions and organisms, and it is considered as a mechanism for short-term storage \cite{galloway-woo-lu2008,major-tank2004,zylberberg-strowbridge2017}. This is one of our main interests here.

Before precising our contributions, we briefly recall some of the main different efforts of research concerning biological neural networks.
On one side of the spectrum, \textit{conductance-based models} deal in detail with the chemical interchanges between neurons when synapses occur. The Hodgkin-Huxley model is classical among those \cite{hodgkin1952quantitative}, and there has been a number of recent works within this framework, see \cite{baladron2012mean,touboul2014spatially,touboul2014propagation}. In general, these models lead to complex equations, which can be simplified in some specific cases \cite{touboul2012mean}. 
On the other side of the spectrum, there are the \textit{integrate-and-fire models}, Lapicque's model belonging to this large group, see Burkitt for reviews \cite{burkitt2006review,burkitt2006review2}. Considering the fast, stereotyped nature of neuron synapses, these models focus on the important problem of time distribution of spikes, which are seen as instantaneous events. In simple words, the membrane potential is the magnitude taken into account for each neuron, and it is considered to integrate the potential received by other neurons, until it reaches some threshold---usually fixed---which causes the neuron to fire spikes, received by other neurons in the network. In addition, \textit{leaky integrate-and-fire} (LIF) models take into account the decay of the membrane potential of neurons towards its resting state, due to the leak of potential across the cellular membrane when not interacting.

Most of the aforementioned works
model random effects through an external Gaussian process representing the aggregate effect of other neurons affecting the network. This allows Itô calculus to be applied when making computations and deriving results. There are, however, some several limitations regarding the classical integrate-and-fire models. For instance, these models do not take into account, in general, the intrinsic randomness of neurons, which is an essential feature that cannot be ignored, see for example \cite{rolls2010noisy}. Some recent models do include some form of intrinsic randomness by introducing white noise associated to each neuron \cite{delarue2015global}, which provides a partial solution regarding this topic. Another issue lies in the so called \textit{avalanche effect} caused by the presence of a fixed threshold above which the neurons deterministically fire. 
{More specifically, the neuron potential blows-up in finite time in the mean-field limit (see for instance \cite{caceres2011analysis}), which does not properly represent the behaviour of a large population of neurons.}

\textit{Linear-nonlinear-Poisson models} (LNP) take a step forward from LIF models, addressing the aforementioned issues in an effective way. The first LNP model was introduced by Chichilnisky \cite{chichilnisky2001simple}. These models represent the behaviour of spikes using inhomogeneous Poisson processes depending on
the membrane potential. The model presented by Robert and Touboul \cite{robert2016dynamics} belongs to this family and is the main inspiration of our framework. Our model is indeed a modification of the former and belongs to this family as well. LNP models seem to represent neuronal firing more accurately, display a good fit with experimental data \cite{pillow2005prediction,pillow2008spatio}, and avoid the problem of blow-up.
In this context, a well defined limiting behaviour is reached.
As underlined in \cite{robert2016dynamics}, the  associated mean-field limit follows a McKean--Vlasov equation where the dynamics of the process depends on mean quantities of the same process.

For diffusive underlying motions, McKean--Vlasov equations associated to
potentials have been extensively studied in the last decades.
See for instance the seminal paper \cite{BENACHOUR1998173} and many ulterior developments (see \cite{malrieu, tugaut} and references therein).
In most works stemming from statistical physics perspectives, convergence (as $N\to\infty$) of a particle system towards its associated mean-field limit is obtained for fixed times.
A convergence result of this kind is typically intrinsically linked to the asymptotic independence between particles, and is referred to as \textit{propagation of chaos}. See also \cite{villani}
for the use of the Wasserstein distance to prove propagation of chaos in that context. 
Let us note however that in these references, strong results on the speed of convergence towards hydrodynamic limits are usually proven in the case of a unique stationary measure and/or building on strong regularity property of the 
limiting PDE. Similarly, the literature of particle systems undergoing energy-preserving binary local (jump-type) interactions and their respective limiting equation is quite extensive, see for instance the seminal work \cite{kac1956}, and more recently \cite{mischler-mouhot2013}, and the references therein. In this setting, many results regarding convergence to equilibrium and propagation of chaos have been obtained. 

In contrast, the local dynamics of the model of the present article involves simultaneous jumps between 3 or more neurons (particles), with no preserved quantities (even on average). Thus, the limiting nonlinear equation that we obtain is of a different nature. For instance, as we shall see, the limit equation can have \textit{two stationary distributions}, depending on the parameters of the model. This has strong consequences for propagation of chaos and speed of convergence results, making them harder to prove and heavily dependent on model parameters. We remark that multiplicity of stationary distributions is also present in \cite{robert2016dynamics}, but without any rate of convergence under the stationary regime nor selection principle for the limiting stationary distribution. Some results concerning metastability properties of McKean--Vlasov equations having several stationary distributions are considered in \cite{carrillo2020long}.

Consequently, many of the classical techniques used in more physical models cannot be applied directly in the present context. Nevertheless, our propagation of chaos study relies on a recent coupling technique for binary interactions developed in \cite{cortez2016quantitative}, which we adapt and extend to our setting.

\subsection{Original aspect of our model and main contribution} 

Following Robert and Touboul \cite{robert2016dynamics}, we consider both a particle system describing potentials in a finite excitatory network of neurons, and a nonlinear stochastic process approximating the behaviour of a typical neuron in a large network.
Our model describes the action potential of a network of $N$ neurons over time, denoted $\mbf{X}_t = (X^1_t,\ldots,X_t^N)$. Each neuron $X^i_t\in\R_{+}$ is affected by an exponential decay towards its resting potential, set as 0.  Also, some random interactions take place within the network. Those occur when a neuron produces a \textit{spike}---also known as \textit{firing}---: it is instantaneously reset to its resting potential and, at the same time, it randomly selects a fixed number $\kappa \in \N$ of other neurons and gives them an excitatory impulse. Those events take place at a state-dependent rate.
In contrast with \cite{robert2016dynamics}, the dynamics we consider here is \textit{local}, in the sense that, in each firing event, the firing neuron communicates with a fixed number of other neurons ($\kappa$) and not all of them at once, and the magnitude of the spikes is of order 1 and not $1/N$. As a consequence, in the limit equation, the firing mechanism gives rise to a nonlinear jump term, and not a drift term as in \cite{robert2016dynamics}. We chose this modelling assumption firstly because it allows to model arguably more interesting dynamics from a biological point of view (a natural extension for future work is to consider an underlying graph structure between communicating neurons).
Secondly, it exhibits a different (and possibly more meaningful) phase transition concerning the long-time behaviour of the process, which has
a natural physical interpretation (see Section \ref{sec:phase_transition}).
Last but not least, the mathematical techniques needed to quantify neurons decorrelations are different from the ones in \cite{robert2016dynamics} and are possibly
generalizable to a larger class of similar models.

The main structure of our discussion is the following. First, we study the large-time behaviour of the finite network (Theorem \ref{the:wellpsdFinite}). We then characterize the mean-field limit process---which we call \textit{nonlinear process}--- and show its existence and uniqueness (Theorem \ref{the:well_posedness_Z}). We show that there is a simple phase transition in terms of the parameters of the model:
if the transmission of potential is strong enough (in terms of strength and frequency), then the limiting process is active forever; otherwise, it decays to 0 (Theorem \ref{the:kappa_alpha}).

We then study propagation of chaos: we prove convergence, for fixed time, of the empirical measure of the finite network towards its mean-field limit, when the number of neurons grows to infinity (Theorem \ref{the:MeanFieldApproximation}). To achieve this, a coupling argument is developed, in the same spirit as in \cite{cortez2016quantitative}. We start by defining a coupling $\mathbf{Z}_t = (Z_t^1,\ldots,Z_t^N)$ such that each $(Z_t^i)_t$ is a nonlinear process. In a second step, we define a coupling ``close'' to the first one, but with independent coordinates. Then, some computations with both couplings allow us to prove convergence in terms of the Wasserstein distance and to obtain interesting rates of convergence.

Finally, we show the following persistence phenomenon (Theorem \ref{the:persist}):
\begin{itemize}
	\item 
	if the transmission of potential is strong enough, then, even though the finite network
	activity dies in finite time for any $N$, it remains active at least for a time of order $\log(N)$.
	
	\item otherwise, the networks activity dies out at a speed independent of $N$.
\end{itemize}

The rest of the paper is organised as follows.
Section \ref{sec:model} describes the interacting neuron model and its hydrodynamic limit characterised by a McKean--Vlasov equation. Section \ref{sec:results} presents the main results: the behaviour of the network for finite $N$, the well-posedness of the limiting stochastic differential equation, the phase transition exhibited by the limiting process in terms of global parameters representing the intensity of connection,
the convergence as $N\to\infty$ of the empirical measure of the finite network towards the mean-field, and the persistence result. Section \ref{sec:proofs} gives detailed proofs. Finally, in Section \ref{sec:future_work} we mention possible further lines of research that might stem from this work.

\section{Model description}
\label{sec:model}

A system of $N$ neurons is modelled as an interacting particle system where the potential of each neuron is represented as a stochastic process with values in $\R_+ = [0,\infty)$. We present three viewpoints to this particle system: piecewise deterministic Markov process (Section \ref{sec:PiecewiseDeterministicMP}), a solution of a stochastic differential equation driven by Poisson noise (Section~\ref{sec:PoissonSDE}), and a discrete-time Markov chain obtained by sampling the system at firing instants (Section~\ref{sec:DTMC}). Thereafter, Section~\ref{sec:MeanField} presents a simplified McKean--Vlasov stochastic differential equation which is used to model systems with a large number of neurons. The model is characterised by four parameters:

\begin{itemize}
	\item decay rate $\mu>0$,
	
	\item firing rate proportionality constant $\gamma>0$,
	
	\item impulse range $\kappa\in\N$ (number of neurons excited at a firing event),
	
	\item impulse magnitude $\rho>0$.
\end{itemize}

In what follows, we always assume that the initial condition $\mbf{X}_0 = (X_0^1,\ldots,X_0^N)$ is a collection of i.i.d.\ copies of a given random variable $Z_0 \in \R_+$. Consequently, $\mbf{X}_t = (X_t^1,\ldots,X_t^N)$ is exchangeable for all $t\geq 0$, in the sense that its distribution is invariant with respect to permuting the neuron indices.

\subsection{Definition as a piecewise deterministic Markov process}
\label{sec:PiecewiseDeterministicMP}

Consider a system of $N$ neurons having potential levels $X^1_t, \dots, X^N_t \ge 0$ at time $t$. Level zero indicates a resting potential. Each neuron has a decay rate $\mu > 0$ and a firing rate $\gamma > 0$. In the absence of firings, the level of each neuron decays according to $\frac{d}{dt} X^i_t = - \mu X^i_t$.  Firings of neuron $i$ are triggered at the events of an inhomogeneous Poisson process with state-dependent intensity $\gamma X^i_t$, independently of other neurons.  When a neuron fires, it jumps to zero, and it selects $\kappa \ge 1$ other neurons uniformly at random, which instantly increase their levels by $\rho > 0$ units. 
We note that, unlike in \cite{robert2016dynamics}, there is no scaling (in $N$) of any of the model parameters $\kappa, \rho, \mu, \gamma$.

The trajectory of $\mathbf{X}_t = (X^1_t, \dots, X^N_t)$ is a piecewise deterministic continuous-time Markov process \cite{Davis_1984} on state space $[0,\infty)^N$ with flow 
$
(\mbf{x},t) \mapsto \mbf{x} e^{-\mu t},
$ 
jump rate
$
\mbf{x} \mapsto \gamma \norm{\mbf{x}}
$ 
where $\norm{\mbf{x}} = \sum_i \abs{x_i}$, and jump kernel
\[
(\mbf{x},A) \mapsto \sum_i \sum_{K: K \not\ni i} \frac{x^i}{\norm{\mbf{x}}} \binom{N-1}{\kappa}^{-1} \delta_{\mbf{x}-x^i e_i + \rho e_K}(A),
\]
where the second sum on the right is over subsets $K \subset \{1,\dots, N\}$ of size $\kappa$,
$e_i$ denotes the $i$-th unit vector and $e_K = \sum_{j\in K} e_j$.

\subsection{Description as a solution of a stochastic differential equation}
\label{sec:PoissonSDE}

Alternatively, the system trajectory relative to a given initial state can be represented as the a.s.\ unique strong solution of a system of stochastic differential equations. Hence, for any $i=1,\ldots,N$, the behaviour of $X_t^i$ can be represented by
\begin{equation}
	\label{eq:dXti2ndVersion}
	\begin{aligned}
		dX_t^i
		&= -\mu X_{t^\-}^i dt - X^i_{t^\-} \sum_{K: K \not\ni i} \int_0^\infty \ind{u \le \gamma X^i_{t^\-}} \cN_{K}(dt,du,(i-1,i]) \\
		&\quad + \rho \sum_{K: K \ni i} \int_0^\infty \int_0^N \ind{u \le \gamma X^{\lceil\xi\rceil}_{t^\-}} \cN_{K}(dt,du,d\xi),
	\end{aligned}
\end{equation}
in which all sums involving $K$ are taken with respect to subsets of $\{1,\dots,N\}$ of size $\kappa$, and $\cN_{K}(dt,du,d\xi)$, $K \subset \{1,\dots, N\},$ are mutually independent Poisson random measures on $\R_+ \times \R_+\times (0,N]$ with common intensity measure
\[
\binom{N-1}{\kappa}^{-1} dt \, du \,  \ind{\lceil\xi\rceil\notin K}d\xi.
\]
Any atom $(T,U,\Xi)$ of the Poisson random measure $\cN_{K}$ corresponds to a possible firing where the neurons in $K$ are excited: $T \ge 0$ is the time of the jump, the mark $U \ge 0$ is used to model inhomogeneities, and $\Xi \in (0,N]$ is used to select uniformly the neuron $i = \lceil\Xi\rceil \notin K$ that fires.

\subsection{Construction using a discrete-time embedded Markov chain}
\label{sec:DTMC}

We now give a third construction for our model using an embedded Markov chain, which is especially well suited for simulating the network.

A system started at state $\mbf{x} = (x^1,\dots, x^N)$ evolves as follows. 
Observe first that the total firing rate in state $\mbf{X}_t$ equals $\gamma \norm{\mbf{X}_t}$, and as long as firings do not occur, the sum of potentials decays according to $\norm{\mbf{X}_t} = \norm{\mbf{x}} e^{-\mu t}$. Therefore, the probability that there are no firings during $[0,t]$ is the same as the probability that an inhomogeneous Poisson point process with intensity $\lambda_t = \gamma \norm{\mbf{x}} e^{-\mu t}$ has no jumps during $[0,t]$. Because the number of such jumps during $[0,t]$ is Poisson distributed with mean $\int_0^t \lambda_s ds = (\gamma/\mu) \norm{\mbf{x}} (1-e^{-\mu t})$, we see that the distribution of the first firing time $\tau$ is characterised by
\begin{equation}
	\label{eq:FirstFiringTime}
	\pr(\tau > t)
	\weq e^{ - (\gamma/\mu) \norm{\mbf{x}} (1-e^{-\mu t}) },
	\qquad 0 \le t < \infty.
\end{equation}
In particular, the system never fires with probability $\pr(\tau = \infty) = e^{ - (\gamma/\mu) \norm{\mbf{x}}}$. Another key observation is that the proportions of neuron levels remain constant in the absence of firings. Therefore, given $\tau < \infty$, the label of the first firing neuron $I$ is distributed according to
\begin{equation}
	\label{eq:FirstFiringNeuron}
	\pr(I = i) \weq \frac{x^i}{\norm{\mbf{x}}},
	\qquad i=1,\dots,N.
\end{equation}

Using \eqref{eq:FirstFiringTime}--\eqref{eq:FirstFiringNeuron}, we can formally define the trajectory of a system starting at $\mbf{X}_0 \in \R_+^N$ as follows. We denote by $\mbf{\hat X}_k$ the state of the system immediately after the $k$-th firing, and by $\hat T_k \in [0,\infty]$ the time of the $k$-th firing, with $\hat T_k = \infty$ indicating that there are fewer than $k$ firings in total. The random variables $(\mbf{\hat X}_k, \hat T_k)$, $k=0,1,\dots$, are recursively defined as follows, starting with $(\mbf{\hat X}_0, \hat T_0) = (\mbf{X}_0,0)$. For any $k \ge 1$, set $\mbf{x} = \mbf{\hat X}_{k-1}$ and sample the $k$-th waiting time $\tau_k$ from distribution \eqref{eq:FirstFiringTime}. If $\tau_k = \infty$, set $(\mbf{\hat X}_k, \hat T_k) = (0,\infty)$. Otherwise sample the label of the $k$-th firing neuron $I_k$ from distribution \eqref{eq:FirstFiringNeuron}, and then sample a set of excited neurons $K_k$ uniformly at random from the subsets of $\{1,\dots,N\} \setminus \{I_k\}$ of size $\kappa$. Then set
\begin{equation}
	\label{eq:SystemDT}
	\big( \mbf{\hat X}_k, \, \hat T_k \big)
	\weq \big( (\mbf{x} - x^{I_k} e_{I_k})e^{-\mu \tau_k} + \rho e_{K_k}, \ \hat T_{k-1} + \tau_k \big),
\end{equation}
The full continuous-time trajectory of the system is then defined by
\begin{equation}
	\label{eq:SystemCT}
	\mbf{X}_t
	\weq \mbf{\hat X}_k e^{-\mu (t - \hat T_k)}
	\qquad \text{for} \quad \hat T_k \le t < \hat T_{k+1}.
\end{equation}
Equations \eqref{eq:SystemDT}--\eqref{eq:SystemCT} provide a simple way to simulate system trajectories. Sampling from the distribution \eqref{eq:FirstFiringTime} can be done by $\tau = -\frac{1}{\mu} \log (  1 - \frac{\mu}{\gamma \norm{\mbf{x}}} \xi)_+$ where $\xi$ is exponentially distributed with mean one and $\log 0 = -\infty$, so that $\tau=\infty$ when $\xi \ge \frac{\gamma \norm{\mbf{x}}}{\mu}$. We also find that $\lim_{t \to \infty} \mbf{X}_t = 0$ if and only if the total number of firings is finite.


\subsection{Mean-field process}
\label{sec:MeanField}

When the number of neurons $N$ is large, we may heuristically derive a mean-field approximation of the system as follows.
Fix a reference neuron $i$. In the absence of firings, the neuron level decays exponentially according to $\frac{d}{dt} X^i_t = - \mu X^i_t$, and the neuron fires at rate $\gamma X^i_t$. When any other neuron fires, neuron $i$ is excited with probability $\binom{1}{1} \binom{N-2}{\kappa-1} \binom{N-1}{\kappa}^{-1} = \frac{\kappa}{N-1}$. Assuming that the empirical sample average $\frac{1}{N}\sum_j X^j_t$ of the neuron levels is close to the mathematical expectation $\E [X^i_t]$, the average rate at which neuron $i$ is excited is approximately
\[
\gamma \frac{\kappa}{N-1}
\sum_{j \ne i} X^j_t 
\wapprox \gamma \kappa \E [X^i_t].
\]
This suggests that the evolution of a reference neuron in a large system can be approximated by the solution of the McKean-Vlasov equation
\begin{equation}
	\label{eq:dZt}
	\begin{aligned} 
		dZ_t
		&\weq - \mu Z_t dt
		- Z_{t^\-} \int_0^\infty \ind{u \le \gamma Z_{t^\-}} \cN(dt,du) \\
		&\quad + \rho \int_0^\infty \ind{u \le \gamma \E [Z_{t^\-}]} \cM(dt,du),
	\end{aligned}
\end{equation}
in which $\cN(dt,du)$ and $\cM(dt,du)$ denote independent Poisson random measures on $\R_+^2$ with intensities $dt du$ and $\kappa dt du$, respectively. 

Note that the third term in \eqref{eq:dZt}, which results from the local firing mechanism, is of jump-type. As mentioned before, this is a key difference from the model in \cite{robert2016dynamics}, in which the firing mechanism is global and gives rise to a drift term. Notice also that the rate at which $Z_t$ jumps upwards depends on its own expected value; it is in this sense that the process is nonlinear.

\section{Main results}
\label{sec:results}

\subsection{Long-term behaviour of the finite-neuron system}

Theorem \ref{the:wellpsdFinite} below tells us that in a finite system of neurons, all neuron potentials converge to zero, regardless of initial state and model parameters. This result also implies that the unique invariant probability distribution of the $\R_+^N$-valued Markov process is the Dirac measure at $\mbf{0} \in \R_+^N$. However, the nature of this convergence will vary drastically, depending on the parameter values.
For future reference, we define a reproduction number
\begin{equation}
	\label{eq:ReproductionNumber}
	\theta
	\weq \kappa (1-e^{-\rho \gamma /\mu}),
\end{equation}
which will play a crucial role in our study. See Section \ref{sec:phase_transition} for a discussion on the interpretation of this quantity.

\begin{theorem}[Long-term behaviour of the finite network]
	\label{the:wellpsdFinite}
	For any initial state, the process $\mbf{X}_t$ converges according to $\lim_{t \to \infty} \mbf{X}_t = \mbf{0}$ almost surely.
	Moreover, when $\theta<1$, the convergence of any particular neuron $i$ occurs exponentially fast (independent of $N$) according to
	\begin{equation}
		\label{eq:wellpsdFinite}
		\E [\omega(X_t^i, 0)]
		\wle \E [\omega(X_0^i, 0)] \, e^{-(1-\theta)\mu t}
	\end{equation}
	for the metric
	\begin{equation}
		\label{eq:Metric}
		\omega(x,y)
		\weq 1 - e^{- (\gamma/\mu)\abs{x-y}}.
	\end{equation}
\end{theorem}

The global decay property described in Theorem~\ref{the:wellpsdFinite} is remarkable because it is valid for arbitrarily large values of the firing rate $\gamma$.
To appreciate this feature, we note that in a slightly modified system in which the potentials of firing neurons remain constant instead of being reset to zero, the associated Markov process $(\mbf{Y}_t)_{t\geq 0}$ satisfies $\frac{d}{dt} \E [\norm{\mbf{Y}_t}] = (\rho\kappa\gamma - \mu) \E [\norm{\mbf{Y}_t}]$, so that the modified system remains bounded in mean only for $\gamma < \frac{\mu}{\rho\kappa}$.

\subsection{Phase transition in the mean field}
\label{sec:phase_transition}

As the number of neurons $N \to \infty$, one expects that any single neuron in the network converges to a nontrivial limit process driven by \eqref{eq:dZt}.
The well-posedness of the SDE \eqref{eq:dZt} representing the mean-field approximation is confirmed by the following result.

\begin{theorem}[Well-posedness of the mean-field SDE]
	\label{the:well_posedness_Z}
	For any initial state $Z_0 \in \R_+$, there exists a unique strong solution $(Z_t)_{t \ge 0}$ to \eqref{eq:dZt}.
\end{theorem}

The main result of this work is the following description of a phase transition of the mean-field limiting process, characterised by the reproduction number $\theta$ defined by \eqref{eq:ReproductionNumber}.

\begin{theorem}[Phase transition in the mean-field limit]
	\label{the:kappa_alpha}
	Assume that $\E [Z_0] > 0$ and $\E[Z_0^2] < \infty$. Then $(Z_t)_{t \ge 0}$ is uniformly integrable, and the following phase transition holds:
	\begin{itemize}
		\item
		If $\theta < 1$, then $\E[Z_t] \to 0$ and $\int_0^\infty \E[Z_t] \, dt < \infty$.
		
		\item If $\theta = 1$, then
		$\E[Z_t] \to 0$ and $\int_0^\infty \E [Z_t] \, dt = \infty$.
		
		\item If $\theta > 1$, then $\inf_{t \ge 0} \E[Z_t] >0$.
	\end{itemize}
	Moreover, for $\theta < 1$, we have $Z_t \to 0$ almost surely, and this convergence is exponentially fast in the metric $\omega$ defined by \eqref{eq:Metric}, according to
	\begin{equation}
		\label{eq:Erho}
		\E [\omega(Z_t,0)] 
		\wle \E [\omega(Z_0,0)] e^{-(1-\theta)\mu t}.
	\end{equation}
\end{theorem}

Thus, \textbf{the mean field exhibits nontrivial long-run behaviour if and only if $\theta > 1$}. A heuristic explanation for this is the following. 
In analogy with \eqref{eq:FirstFiringTime}, we observe that $p = 1-e^{-\rho\gamma/\mu}$ equals the probability that a neuron which gets excited from the resting potential will eventually fire. In a large system where initially an overwhelming majority of neurons are resting, the number of eventually firing neurons excited by a firing neuron is approximately binomially distributed with success probability $p$ and mean $\theta = \kappa p$. An associated Galton--Watson branching process may survive forever if and only if $\theta > 1$. Therefore:
\begin{itemize}
	\item When $\theta >1$, the resets and firings have enough frequency and power to sustain themselves forever.
	
	\item If $\theta < 1$, then fewer and fewer firings occur as time grows, and eventually the process dies out due to the exponential decay or a final reset.
	
	\item Interestingly enough, in the critical case $\theta=1$, even though $\E[Z_t]\to 0$, the process $Z_t$ never \textit{really} dies: firings do become more infrequent with time, but, since $\int_0^\infty \E[Z_t] dt = \infty$, they never actually stop.
\end{itemize}

\begin{remark}
	\label{rmk:comparison_phase_transition_RT}
	The model of Robert and Touboul also exhibits a similar phase transition, but the condition at which it takes place is different from ours. Indeed, in their model, for the case of a linear firing-rate function, the phase transition occurs when the quantity $\theta_c = \kappa \rho \gamma/\mu$ exceeds 1, see \cite[Section 7.2.1]{robert2016dynamics} (we used the notation of the present paper; note that, because the parameter $\kappa$ has no equivalent in their setting, we have interpreted $\kappa \rho$ as the magnitude of the excitatory impulse in their model, before rescaling by $1/N$). Because $\theta \le \theta_c$, we see that persistence in our model requires a higher value of $\rho \gamma /\mu$, i.e., the magnitude of the excitatory impulse times the ratio of firing rate to decay rate. This difference is a consequence of the local nature of the firing mechanism we consider here, and it is the most distinctive and interesting feature of our model. However, our upper bounds for the finite network (Theorems~\ref{the:MeanFieldApproximation} and~\ref{the:persist}) indicate that $\theta_c$ might have have relevance also for the present model.
\end{remark}

\begin{remark}
	\label{remark:conjecture}
	Theorem~\ref{the:kappa_alpha} implies that if $\E[Z_0]>0$ and $\theta>1$, then any weak limit of $Z_t$ satisfies $\E[Z_\infty] > 0$. We conjecture that for each $\theta>1$, the process $Z_t$ admits exactly two stationary distributions:
	the trivial one (Dirac mass at 0), which is reached only when
	$Z_0=0$ almost surely, and a second, nontrivial distribution,
	towards which $Z_t$ approaches whenever $\E[Z_0]>0$. This nontrivial stationary distribution could be computed through a fixed point procedure.
	The construction in Section~\ref{sec:DTMC} gives a simple way to approximately simulate this distribution.
\end{remark}

As a consequence of Theorem \ref{the:kappa_alpha}, we obtain the following result regarding the
limiting proportion of neurons at the resting potential in the infinite network.

\begin{corollary}
	\label{cor:kappa_alpha}
	The process $Z_t$ satisfies for all $t \ge 0$
	\begin{equation}
		\label{eq:cor:kappa_alpha}
		\PP(Z_t=0)
		\weq \frac{1}{\kappa} + \left( \PP(Z_0=0) - \frac{1}{\kappa}\right) e^{-\gamma \kappa \int_0^t \E[Z_s] ds}.
	\end{equation}
	Consequently, $\lim_{t\to \infty} \PP(Z_t=0) = \frac{1}{\kappa}$ when $\E[Z_0]>0$ and $\theta \ge 1$.
\end{corollary}

\subsection{Approximation using the mean field}
\label{sec:Approximation}

The following result confirms that the behaviour of a typical neuron is approximated by the nonlinear mean-field SDE \eqref{eq:dZt}, when the number of neurons $N \to \infty$.
The upper bounds on convergence rate are expressed in terms of another fundamental constant,
\begin{equation}
	\label{eq:ThresholdChaos}   
	\theta_c = \kappa \rho \gamma/\mu,
\end{equation}
which dominates the reproduction number defined by \eqref{eq:ReproductionNumber} according to $\theta \le \theta_c$.
See also Remark~\ref{rmk:comparison_phase_transition_RT}. The subscript $c$ in $\theta_c$ stands for ``chaos''.

We recall that, given probability measures $\nu, \tilde{\nu}$ on $\R$, their \textit{$1$-Wasserstein distance} is defined as
\[
W_1(\nu, \tilde{\nu})
= \inf \E[ |Y-\tilde{Y}|],
\]
where the infimum is taken over all \textit{couplings} of $\nu$ and $\tilde{\nu}$, that is, over all possible random pairs $(Y,\tilde{Y})$ where $Y \sim \nu$ and $\tilde{Y} \sim \tilde{\nu}$. $W_1$ is a complete metric on the space of probability measures with finite first moment.  Notice that \textit{any} coupling immediately provides an upper estimate for $W_1$; we will use this fact often.

\begin{theorem}[Propagation of chaos]
	\label{the:MeanFieldApproximation}
	Assume that $\kappa = 2$ and $\E[(X_0^1)^2]<\infty$. Then there exist a constant $C>0$ such that for all $N$ and all $t$, the expected Wasserstein-1 distance between $\bar{\mathbf{X}}_t = \frac{1}{N} \sum_{i=1}^N \delta_{X^i_t}$ and $f_t = \law(Z_t)$ satisfies
	\[
	\E \left[ W_1(f_t, \bar{\mathbf{X}}_t) \right]
	\wle \frac{C}{N^{1/3}}
	\times
	\begin{cases}
		1,   &\quad \theta_c < 1, \\
		1+t, &\quad \theta_c = 1, \\
		e^{(\theta_c-1) \mu t}, &\quad \theta_c > 1.
	\end{cases}
	\]
\end{theorem}

Theorem~\ref{the:MeanFieldApproximation} proves
propagation of chaos in terms of the expected $W_1$-distance between the empirical measure of the finite network and its mean-field limit. We mention that, besides convergence as $N\to\infty$, the term ``propagation of chaos'' typically also refers to the asymptotic independence between particles in the system, but this is equivalent to the convergence of its empirical measure, see for instance the classical notes of Sznitman \cite{sznitman1989}. Rate $N^{-1/3}$ in the upper bound appears reasonable,
considering that the convergence of the empirical measure of a sequence of i.i.d.\ variables towards their common law occurs at rate $N^{-1/2}$ when measured in expected $W_1$-distance, see \cite{fournier2015rate} for details. The rate is uniform in time in the case $\theta_c < 1$.

Theorem~\ref{the:MeanFieldApproximation} is proved in Section~\ref{sec:ProofMeanFieldApproximation} using a coupling argument introduced in \cite{cortez2016quantitative}. To this end,
we will write the SDE \eqref{eq:dXti2ndVersion} of a neuron in the finite network, and the mean-field SDE \eqref{eq:dZt}, in a slightly different way, which is easier to handle and displays similarities between the two processes in a natural fashion.

\begin{remark}
	\label{remark:general_kappa}
	For an arbitrary $\kappa \in \N$ fixed, or even a random $\kappa \in \{1,\ldots,N-1\}$ with fixed mean and bounded variance, the main idea of the proof still works. However, the full proof requires significant additional notation and a more careful treatment, rendering the argument unnecessarily technical and difficult to follow, while adding little to no new insights. Thus, we decided to treat only the case $\kappa = 2$.
\end{remark}

\subsection{Persistence}

We can now combine the previous results and state the following theorem that sheds light on the phenomenon of persistence, in terms of the reproduction number $\theta$ defined in \eqref{eq:ReproductionNumber} and the threshold parameter $\theta_c$ in \eqref{eq:ThresholdChaos}.

\begin{theorem}[Persistence]
	\label{the:persist}
	Assume that $\kappa=2$ and $\E[(X_0^1)^2]<\infty$.
	If $\E[X_0^1]>0$, then the finite network satisfies the following dichotomy:
	\begin{itemize}
		\item If $\theta>1$, then there exist constants $c >0$, $\tilde{c} \in \R$, and $\epsilon>0$ (not depending on $N$) such that for all $N \ge 1$,
		\[
		\E [X_t^1]
		\wge \epsilon
		\quad
		\text{for all} \ t \le \tilde{c} + c \log N.
		\]
		\item If $\theta<1$, then for the metric $\omega$ defined by \eqref{eq:Metric}, we have
		\[
		\E [\omega(X_t^1, 0)]
		\wle \E [\omega(X_0^1, 0)] \, e^{-(1-\theta)\mu t}.
		\]
		If we additionally assume that $\theta_c < 1$, then
		\[
		\E [ X_t^1 ]
		\wle  \E [ X_0^1 ]
		e^{-(1-\theta_c)\mu t}.
		\]
	\end{itemize}
\end{theorem}

Theorem \ref{the:persist}
states that in the case $\theta>1$, the activity of the finite network persists for a time of order (at least) $\log N$; but if $\theta<1$, then the activity dies out at a rate independent of $N$. 

\begin{remark}
	\label{remark:logN}
	We believe that $\log N$ is a crude lower estimate on the time it takes for the activity to decay. Numerical simulations suggest that when $\theta>1$, the activity of the network persists much longer. Moreover, for $N$ mildly large and $\theta>1$ not too close to 1, the empirical measure of the network appears to stabilize around some nontrivial distribution; this suggests that the decay time grows very fast with $N$, possibly exponentially. This behaviour can also be observed in \cite[Section 7.2.1]{robert2016dynamics}, although the condition at which the phase transition occurs is different, see Remark \ref{rmk:comparison_phase_transition_RT}.
\end{remark}

\section{Proofs} 
\label{sec:proofs}

In this section we provide detailed proofs for all of our results.
By measuring neuron potentials in units of firing amplitude, we may set $\rho=1$, and by measuring time in units of a neuron half life divided by $\log(2)$, we may set $\mu=1$.
To simplify notation, we will hence often set $\rho=\mu=1$ without
loss of generality.

We first obtain a simple formula for the expectation of a function of the process $(\mbf{X}_t)_{t\geq 0}$, which will be used several times.
In general fix $\phi: \R^N \to \R$, and let $F_t = \phi(\mbf{X}_t)$.
In the absence of firings, $F_t$ evolves according to $\frac{d}{dt} F_t =
\sum_i \partial_i \phi(\mbf{X}_t) \frac{d}{dt} X_t^i = - \mu \sum_i X_t^i \partial_i \phi(\mbf{X}_t)$. If neuron $i$ fires and excites a set neurons $K$ at time instant $t$, then $F_t$ changes according to $F_t - F_{t^\-} = \phi(\mbf{X}_{t^\-} + \rho \sum_{j \in K} e_j - X_{t^\-}^i e_i) - \phi(\mbf{X}_{t^\-})$. It follows that
\begin{align*}
	d F_t
	&\weq - \mu \sum_i X_t^i \partial_i \phi(\mbf{X}_t) dt \\
	&\quad + \sum_i \sum_{K \not\ni i} \int_{\R_+} \ind{u \le \gamma X^i_{t^\-}}
	\Big( \phi(\mbf{X}_{t^\-} + \rho e_K - X_{t^\-}^i e_i) - \phi(\mbf{X}_{t^\-}) \Big) \cN_{i,K}(dt,du).
\end{align*}
Taking expectations, we find that
\begin{equation}
	\label{eq:DriftGeneral}
	\begin{aligned}
		\frac{d}{dt} \E [ F_t ]
		&\weq - \mu \sum_i \E \left[X_t^i \partial_i \phi(\mbf{X}_t)\right] \\
		&\quad + \binom{N-1}{\kappa}^{-1} \ \sum_i \sum_{K \not\ni i} \gamma \E \left[X^i_t
		\Big( \phi(\mbf{X}_{t} + \rho e_K - X_{t}^i e_i) - \phi(\mbf{X}_{t}) \Big)\right]
	\end{aligned}
\end{equation}

\subsection{Long-run behaviour of finite-neuron system (Theorem~\ref{the:wellpsdFinite})}

\begin{proof}[Proof of Theorem~\ref{the:wellpsdFinite}]
	Strong existence of solutions to \eqref{eq:dXti2ndVersion} is obtained by truncating the firing rate and letting the truncation parameter approach $\infty$, as in the proof of Lemma~\ref{lemma:lemmaZ}. Uniqueness is standard.
	
	We will next verify that $\norm{\mbf{X}_t} = \sum_{i=1}^N X^i_t \to 0$ almost surely.
	We claim that $H_t = e^{- (\gamma/\mu) \norm{\mbf{X}_t}}$ is a submartingale with respect to the filtration $(\cF_t)_{t \ge 0}$ generated by the marked point processes $(\cN^i)_{i=1}^N$, where $\cN^i:=\sum_{K: K \not\ni i} \cN_{K}(dt,du,(i-1,i])$. In the absence of firings, each neuron evolves according to $\frac{d}{dt} X^i_t = -\mu X^i_t$, so that $\frac{d}{dt} \norm{\mbf{X}_t} = -\mu \norm{\mbf{X}_t}$ and $\frac{d}{dt} H_t = - (\gamma/\mu) H_t \frac{d}{dt} \norm{\mbf{X}_t} = \gamma \norm{\mbf{X}_t} H_t$.
	If neuron $i$ fires at time instant $t$, we see that $\norm{\mbf{X}_t} = \norm{\mbf{X}_{t^\-}} + \rho\kappa - X^i_{t^\-}$, and it follows that
	\[
	H_t - H_{t^\-}
	\weq \left( e^{-(\gamma/\mu)(\rho\kappa - X^i_{t^\-})} - 1 \right) H_{t^\-}.
	\]
	It follows that for any $0 \le s \le t$,
	\[
	\begin{aligned}
		H_t - H_s
		&\weq \int_s^t \gamma \norm{\mbf{X}_r} H_r \, dr \\
		&\qquad + \sum_i \int_s^t \int_{\R_+}
		\ind{u \le \gamma X_{r^\-}^i}
		\left( e^{-(\gamma/\mu)(\rho\kappa - X^i_{r^\-})} - 1 \right) H_{r^\-} \cN^i(dr,du),
	\end{aligned}
	\]
	from which we conclude that
	\begin{equation}
		\label{eq:proof_prop_wellpsd_ps_Ht}
		H_t - H_s
		\wge \int_s^t \gamma \norm{\mbf{X}_r} H_r \, dr \\
		- \sum_i \int_s^t \int_{\R_+}
		\ind{u \le \gamma X_{r^\-}^i} H_{r^\-} \cN^i(dr,du)
	\end{equation}
	almost surely. By noting that $\sum_i \int_{\R_+} \ind{u \le \gamma X^i_r} du = \gamma \norm{\mbf{X}_r}$, and recalling that the intensity measure of $\cN^i(dr,du)$ equals $dr \, du$, we find that
	\[
	\E\left[ \sum_i \int_s^t \int_{\R_+}
	\ind{u \le \gamma X_{r^\-}^i} H_{r^\-} \cN^i(dr,du) \, \middle| \, \cF_s \right]
	\weq \E\left[ \int_s^t \gamma \norm{\mbf{X}_r} H_r \, dr \, \middle| \, \cF_s \right].
	\]
	The conditional expectation of the right side of \eqref{eq:proof_prop_wellpsd_ps_Ht} given $\cF_s$ hence equals zero, and it follows that $\E [H_t - H_s \, | \, \cF_s] \ge 0$ almost surely. Hence $H_t$ is a bounded submartingale, and by Doob's martingale convergence theorem, $H_t \to H_\infty$ almost surely, where $H_\infty \in [0,1]$ is a random variable. Hence $\norm{\mbf{X}_t} \to - \frac{\mu}{\gamma} \log H_\infty \in [0,\infty]$ almost surely.
	
	Similarly, we find that
	\[
	\begin{aligned}
		\norm{\mbf{X}_t} - \norm{\mbf{X}_s}
		&\weq - \int_s^t \mu \norm{\mbf{X}_r} \, dr \\
		&\qquad + \sum_i \int_s^t \int_{\R_+}
		\ind{u \le \gamma X_{r^\-}^i}
		(\rho\kappa - X^i_{r^\-}) \cN^i(dr,du).
	\end{aligned}
	\]
	By taking expectations, we see that
	\begin{align*}
		\E [\norm{\mbf{X}_t}] - \E [\norm{\mbf{X}_s}]
		&\weq - \int_s^t \mu \E [\norm{\mbf{X}_r}] \, dr
		+ \E \left[\sum_i \int_s^t \gamma X_{r}^i 
		(\rho\kappa - X^i_{r}) dr\right],
	\end{align*}
	so that
	\[
	\frac{d}{dt} \E [\norm{\mbf{X}_t}]
	\weq (\gamma\rho\kappa - \mu) \E [\norm{\mbf{X}_t}]
	- \gamma \E \left[\sum_i (X_{t}^i)^2\right].
	\]
	By Jensen's inequality,
	$\E \left[\sum_i (X_t^i)^2\right] \ge N^{-1} \E \left[\norm{\mbf{X}_t}^2\right] \ge N^{-1}(\E [\norm{\mbf{X}_t}])^2$, and it follows that $y_t = \E [\norm{\mbf{X}_t}]$ satisfies
	\[
	\frac{d}{dt} y_t
	\wle (\gamma\rho\kappa - \mu) y_t
	- \gamma N^{-1} y_t^2.
	\]
	The above inequality shows that $\frac{d}{dt} y_t \le 0$ whenever $y_t \ge N (\rho\kappa - \mu/\gamma)$. Therefore, $\sup_{t \ge 0} y_t < \infty$.  Fatou's lemma implies that $\E [\lim_{t \to \infty} \norm{\mbf{X}_t}] \le \liminf_{t \to \infty} \E [\norm{\mbf{X}_t}] < \infty$. Especially, $\lim_{t \to \infty} \norm{\mbf{X}_t} < \infty$ almost surely. Furthermore, because $0 \in \R_ +^N$ is the only fixed point of the deterministic evolution $\frac{d}{dt} \norm{\mbf{X}_t} = - \mu \norm{\mbf{X}_t}$, and because the number of firings on every bounded time interval is finite almost surely, we conclude that $\lim_{t \to \infty} \norm{\mbf{X}_t} = 0$ almost surely.

	Assuming $\theta < 1$, we will now prove \eqref{eq:wellpsdFinite}.
	By applying \eqref{eq:DriftGeneral} to $\phi(\mbf{x}) = e^{-(\gamma/\mu)x_i}$, noting that $\binom{N-1}{\kappa}^{-1} \binom{N-2}{\kappa-1} = \frac{\kappa}{N-1}$, we find that
	\begin{equation*}
		\begin{aligned}
			\frac{d}{dt} \E \left[e^{-(\gamma/\mu)X^i_t}\right]
			&\weq \gamma \E \left[X_t^i e^{-(\gamma/\mu)X^i_t}\right] \\
			&\quad + \gamma \E \left[X^i_t
			\Big( 1 - e^{-(\gamma/\mu)X^i_t} \Big)\right] \\
			&\quad + \binom{N-1}{\kappa}^{-1} \sum_{j \ne i} \sum_{K \ni i, K \not\ni j} \gamma \E \left[ X^j_t \Big( e^{-(\gamma/\mu)(X^i_t + \rho)} - e^{-(\gamma/\mu)X^i_t} \Big) \right] \\
			&\weq \gamma \E\left[X^i_t\right]
			- \gamma \Big( 1 -  e^{- \rho \gamma/\mu} \Big) \frac{\kappa}{N-1}
			\sum_{j \ne i} \E \left[ X^j_t e^{-(\gamma/\mu)X^i_t} \right] \\
			&\wge  \gamma \E \left[X^i_t\right]
			- \gamma \Big( 1 -  e^{- \rho \gamma/\mu} \Big) \frac{\kappa}{N-1}
			\sum_{j \ne i} \E \left[X^j_t\right] \\
			&\weq \gamma ( 1 - \theta ) \E \left[X^i_t\right],
		\end{aligned}
	\end{equation*}
	thanks to exchangeability.
	In light of $(\gamma/\mu)X^i_t \ge 1 - e^{-(\gamma/\mu)X^i_t}$, it follows that
	\[
	\frac{d}{dt} \E \left[e^{-(\gamma/\mu)X^i_t}\right]
	\wge \mu ( 1 - \theta )
	\Big( 1 - \E \left[e^{-(\gamma/\mu)X^i_t}\right] \Big).
	\]
	Hence, $g_t = 1-\E [ e^{-(\gamma/\mu)X^i_t}] $ satisfies $\frac{d}{dt} g_t \le - \mu(1-\theta) g_t$. The desired inequality \eqref{eq:wellpsdFinite} now follows from Grönwall's Lemma.
\end{proof}

\subsection{Well-posedness of the mean field (Theorem~\ref{the:well_posedness_Z})}

Theorem~\ref{the:well_posedness_Z}
is proved by following the strategy of \cite[Theorem 2]{robert2016dynamics}.
For simplicity, in this section we assume $\rho=\mu=1$.
We first linearise the SDE by fixing a rate function for the firings: given a continuous function $r: \R_+ \to \R_+$, consider the following SDE, with initial condition $Y_0^r = Z_0$:
\begin{equation}
	\label{eq:dYtr}
	dY_t^r
	= - Y_{t^\-}^r dt
	- Y_{t^\-}^r \int_0^\infty \ind{u\le \gamma Y_{t^\-}^r} \cN(dt,du)
	+ \int_0^\infty \ind{u \le \gamma r_t} \cM(dt,du),
\end{equation}
where $\cN$ and $\cM$ are as in \eqref{eq:dZt}.
Denote $\Vert r \Vert_T := \sup_{t\in[0,T]} |r_t|$.

\begin{lemma}
	\label{lemma:lemmaZ}
	There exists a unique strong solution of \eqref{eq:dYtr}.
\end{lemma}

\begin{proof}[Proof of Lemma \ref{lemma:lemmaZ}]
	Uniqueness is classical. To prove existence, fix a time horizon $T>0$, a continuous function $r:[0,T]\to\R_+$, a threshold $M>0$, and consider the truncated SDE
	\begin{equation}
		\label{eq:dYtrM}
		\begin{split}
			dY_t^{r,M}
			&= - Y_{t^\-}^{r,M} dt
			- Y_{t^\-}^{r,M} \int_0^\infty \ind{u\le \min(\gamma Y_{t^\-}^{r,M},M)} \cN(dt,du) \\
			& \qquad\qquad {} + \int_0^\infty \ind{u \le \gamma r_t} \cM(dt,du),
		\end{split}
	\end{equation}
	with $Y_0^{r,M} = Z_0$. Since the rates are bounded, there is strong existence and uniqueness for \eqref{eq:dYtrM}. Consider the event $B_M = \{ \gamma Y_t^{r,M} \le M : \forall t \in[0,T]\}$, which increases with $M$, and notice that for $\tilde{M} \geq M$ we have $Y^{r,M} \equiv Y^{r,\tilde{M}}$ on $B_M$. Thus, we can define $Y_t^r = \lim_{M\to\infty} Y_t^{r,M}$ on the event $B_\infty = \bigcup_{M>0} B_M$. It is easy to see that $(Y_t^r)$ solves \eqref{eq:dYtr} on $B_\infty$. It remains to show that $\PP(B_\infty) = 1$, for which it is enough to prove that $\lim_{M\to\infty} \PP(B_M^c) = 0$. Indeed, since $\Vert r \Vert_T < \infty$, it is not hard to obtain a bound like $\PP(B_M^c) \le C_T/M$ for some constant $C_T>0$ (for instance, study the solution of \eqref{eq:dYtr} without resets and exponential decay and compare with $Y_t^{r,M}$). This proves existence and uniqueness for \eqref{eq:dYtr} on $[0,T]$ for all $T>0$, which can be easily extended to $[0,\infty)$.
\end{proof}

Now consider a functional $\cA : C(\R_+,\R_+) \to C(\R_+,\R_+)$ defined by
\begin{equation}
	\label{eq:cA}
	(\cA r)_t = \E \left[Y_t^r\right]
	\quad \text{where $(Y_t^r)$ solves \eqref{eq:dYtr}}.
\end{equation}

\begin{lemma}
	\label{lemma:LemmaFixedP}
	The functional $\cA$ has a unique fixed point.
\end{lemma}
\begin{proof}[Proof of Lemma \ref{lemma:LemmaFixedP}]
	Fix a time horizon $T>0$, and let $r,q\in C([0,T],\R_+)$. Let $(Y_t^r)$, $(Y_t^q)$ be the strong solutions to \eqref{eq:dYtr} with rate functions $r$ and $q$, respectively, with $Y_0^r = Y_0^q = Z_0$. Call $h_t = \E\left[|Y_t^r-Y_t^q|\right]$, then:
	\[
	\partial_t h_t
	= - h_t - \gamma \E \left[|Y_t^r-Y_t^q|^2\right] + \gamma \kappa |r_t - q_t|
	\le - h_t + \gamma \kappa |r_t - q_t|.
	\]
	Consequently,
	\[
	|(\cA r)_t - (\cA q)_t |
	\le h_t
	\le \gamma \kappa \int_0^t |r_s - q_s| e^{-(t-s)} ds
	\le \gamma \kappa \int_0^t |r_s - q_s| ds.
	\]
	Iterating this inequality gives
	\[
	\Vert \cA^n r - \cA^n q \Vert_T
	\le \frac{(\gamma \kappa T)^n}{n!} \Vert r - q \Vert_T
	\]
	for all $n\in\N$. Thus, for $n$ large enough, $\cA^n$ is a contraction on $(C([0,T],\R_+), \Vert \cdot \Vert_T)$, then $\cA$ has a unique fixed point in $C([0,T],\R_+)$. The extension to $C(\R_+,\R_+)$ is straightforward.
\end{proof}

\begin{proof}[Proof of Theorem \ref{the:well_posedness_Z}]
	Define $Z_t = Y_t^r$, where $(Y_t^r)$ is the unique strong solution to \eqref{eq:dYtr} with rate function $r\in C(\R_+,\R_+)$ chosen as the the unique fixed point of the functional $\cA$ given by \eqref{eq:cA}. By definition of $\cA$, we have $\E[Z_t] = r_t$, thus $(Z_t)$ solves \eqref{eq:dZt}. 
	
	Uniqueness is straightforward: let $Z$ and $\tilde{Z}$ be strong solutions to \eqref{eq:dZt} with $Z_0=\tilde{Z}_0$. If we define $r_t:=\E[Z_t]$ and $\tilde{r}_t:=\E[\tilde{Z}_t]$, then $Z=Y^r$ and $\tilde{Z}=Y^{\tilde{r}}$ are strong solutions to \eqref{eq:dYtr} and is immediate to prove that both $r$ and $\tilde{r}$ are fixed points for $\cA$, so $r=\tilde{r}$ by Lemma \ref{lemma:LemmaFixedP}, and hence $Z=Y^r=Y^{\tilde{r}}=\tilde{Z}$ by Lemma \ref{lemma:lemmaZ}.
\end{proof}

\subsection{Phase transition in the mean field (Theorem~\ref{the:kappa_alpha})}

\begin{proof}[Proof of Theorem~\ref{the:kappa_alpha}]
	Lemma \ref{the:UniformMomentBounds} below implies that $\sup_t \E [Z_t^2] < \infty$ when we assume that $\E [Z_0^2] < \infty$, thus $(Z_t)_{t\geq 0}$ is uniformly integrable.
	
	We now study phase transition.
	Let us analyse the evolution of $h_t = \E \left[e^{-(\gamma/\mu) Z_t}\right]$. Observe that $h_t\in[0,1]$. In absence of firings, $e^{-(\gamma/\mu) Z_t}$ evolves according to $\frac{d}{dt} e^{-(\gamma/\mu) Z_t} = e^{-(\gamma/\mu) Z_t} (-(\gamma/\mu) \frac{d}{dt} Z_t) = \gamma Z_t e^{-(\gamma/\mu) Z_t}$. In a reset, $e^{-(\gamma/\mu) Z_t}$ increases by $1 - e^{-(\gamma/\mu) Z_{t^\-}}$, and in an event of excitation, it increases by $(e^{-\rho \gamma/\mu} - 1) e^{-(\gamma/\mu) Z_{t^\-}}$. Therefore,
	\begin{align*}
		d (e^{-(\gamma/\mu) Z_t})
		&= \gamma Z_t e^{-(\gamma/\mu) Z_t} dt \\
		& \quad {} + ( 1 - e^{-(\gamma/\mu) Z_{t^\-}} ) \int_{\R_+} \ind{u \le \gamma Z_{t^\-}} \cN(dt,du) \\
		& \quad {} + (e^{-\rho\gamma/\mu} - 1) e^{-(\gamma/\mu) Z_{t^\-}} 
		\int_{\R_+} \ind{u \le \gamma \E\left[Z_{t^\-}\right]} \cM(dt,du),
	\end{align*}
	where $\cN(dt,du)$ and $\cM(dt,du)$ are as in \eqref{eq:dZt}. By taking expectations, we find that the terms $\E[\gamma Z_t e^{-(\gamma/\mu) Z_t}]$ cancel, thus obtaining
	\begin{equation}
		\label{eq:dhtlambda}
		\frac{d}{dt} h_t
		\weq \gamma r_t  - \gamma\theta r_t h_t,
	\end{equation}
	where $r_t = \E [Z_t]$, and we recall that $\theta = \kappa (1-e^{\rho\gamma/\mu})$. This differential equation is solved by
	\[
	h_t
	= e^{-\gamma\theta R_t} \,  h_0
	+ \left(1 - e^{-\gamma\theta R_t}  \right) \frac{1}{\theta},
	\]
	where $R_t = \int_0^t r_s ds$. Because $R_t$ is nondecreasing, it follows that $h_t$ converges to
	\begin{equation}
		\label{eq:hinfty}
		h_\infty
		= e^{-\gamma\theta R_\infty}  h_0
		+ \left(1 - e^{-\gamma\theta R_\infty}  \right) \frac{1}{\theta},
	\end{equation}
	where $R_\infty = \int_0^\infty r_s ds \in [0,\infty]$. Recall that $h_t\in[0,1]$ for all $t \geq 0$, so $h_\infty\in [0,1]$ as well. We now prove that $h_\infty = \min(1,\frac{1}{\theta})$. Let us first state some remarks:
	
	\begin{enumerate}
		\item $h_t$ and $h_\infty$ are convex combinations of $h_0<1$ (since  $\E[Z_0]>0$) and $\frac{1}{\theta}$.
		
		\item If $R_\infty = \infty$, then $h_\infty = \frac{1}{\theta}$. The converse is also true when $h_0 \neq \frac{1}{\theta}$ (which is the case if $\theta \leq 1$, because $h_0 < 1 \leq \frac{1}{\theta}$).
		
		\item If $R_{\infty}<\infty$, then  $h_\infty = 1$. This is because $R_{\infty}<\infty$ implies $\lim_{n \to \infty} \E[Z_{t_n}]$ $= 0$ for some sequence $t_n \to \infty$. Hence the inequality $e^{-x} \geq 1-x$ gives $h_\infty = \lim_{n \to \infty} \E[ e^{-(\gamma/\mu) Z_{t_n}}] \geq \lim_{n \to \infty} \E [ 1-(\gamma/\mu) Z_{t_n}] = 1$, so $h_\infty = 1$.   
	\end{enumerate}
	
	Consequently,
	
	\begin{itemize}
		\item When $\theta < 1$, we must have that $R_\infty < \infty$, because otherwise Remark 2 would imply that $h_\infty = \frac{1}{\theta} > 1$. Remark 3 gives $h_\infty = 1$.
		
		\item When $\theta = 1$, then $h_\infty = 1$ regardless of the value of $R_\infty$, thanks to Remarks 2 and 3. But using the converse in Remark 2, we deduce that $R_\infty = \infty$.
		
		\item When $\theta > 1$, then $\frac{1}{\theta}<1$ and Remark 1 implies that $h_\infty < 1$. So, by Remark 3, $R_\infty = \infty$, which in turn implies  $h_\infty = \frac{1}{\theta}$ by Remark 2.
	\end{itemize}
	
	This proves that $h_\infty = \min(1,\frac{1}{\theta})$, along with $\int_0^\infty \E[Z_t] dt < \infty$ when $\theta < 1$ and $\int_0^\infty \E[Z_t] dt = \infty$ when $\theta = 1$. Thus, $\lim_{t \to \infty} \E \left[e^{-(\gamma/\mu) Z_t} \right] = 1$ for $\theta \le 1$, which means that $\E\left[ \omega(Z_t,0)\right] \to 0$ for the metric $\omega(x,y) = 1-e^{-(\gamma/\mu)|x-y|}$. This implies that $1-e^{-(\gamma/\mu)Z_t} \to 0$ in probability, and hence also $Z_t \to 0$ in probability. Because $(Z_t)_{t \ge 0}$ is uniformly integrable, it follows that $\E \left[Z_t\right] \to 0$.
	
	For the case $\theta < 1$, we now prove that $\lim_t Z_t = 0$ almost surely: because $\int_0^\infty r_t dt < \infty$, we deduce that only finitely many atoms $(t,u)$ of $\cM$ satisfy $u \le \gamma r_t$, which implies that the time of the last firing of $Z_t$ is finite almost surely; after that, $Z_t$ will converge to 0 due to the exponential decay or a final reset.
	
	Still assuming that $\theta < 1$, we now prove \eqref{eq:Erho}: noting that $h_t \leq 1$ and $(\gamma/\mu) Z_t \ge 1-e^{-(\gamma/\mu) Z_t}$, from \eqref{eq:dhtlambda} we obtain
	\[
	\frac{d}{dt} h_t
	\geq (1-\theta) \gamma r_t
	\geq (1-\theta) \mu (1-h_t).
	\]
	Therefore, $g_t = 1 - h_t$ satisfies $\frac{d}{dt} g_t \leq -(1-\theta)\mu g_t$, from which it follows that $g_t \leq g_0 e^{-(1-\theta)\mu t}$.
	This confirms \eqref{eq:Erho} and concludes the proof of the claims concerning $\theta < 1$ and $\theta = 1$.
	
	Finally, we treat the case $\theta>1$: using Remark 1, we see that
	\[
	\sup_{t\geq 0} h_t \leq \max\left(h_0, \frac{1}{\theta}\right) < 1.
	\]
	Since $(\gamma/\mu) \E[Z_t] \geq 1 - h_t$, we deduce that $\inf_{t\geq 0} \E[Z_t] >0$. The proof is now complete.
\end{proof}

\begin{proof}[Proof of Corollary \ref{cor:kappa_alpha}]
	Define $p_t = \PP(Z_t=0) = \E[\ind{Z_t=0}]$. Clearly, in a reset event the indicator $\ind{Z_t=0}$ increases by $1 - \ind{Z_{t^\-}=0}$, in a firing it increases by $-\ind{Z_{t^\-}=0}$, while the drift does not affect it. Consequently,
	\[
	\frac{d}{dt} p_t
	= \E [(1 - \ind{Z_t=0}) \gamma Z_t] - \E[\ind{Z_t=0} \gamma \kappa r_t]
	= \gamma r_t - \gamma \kappa r_t p_t.
	\]
	where $r_t = \E[Z_t]$. This differential equation is solved by
	\[
	p_t = \frac{1}{\kappa} + \left( p_0 - \frac{1}{\kappa} \right) e^{-\gamma \kappa \int_0^t r_s ds},
	\]
	which is exactly \eqref{eq:cor:kappa_alpha}. Moreover, thanks to Theorem \ref{the:kappa_alpha}, we know that $\int_0^\infty r_s ds = \infty$ when $\E[Z_0]>0$ and $\theta \geq 1$. Taking limits, we obtain $\lim_{t\to\infty} p_t = \frac{1}{\kappa}$ in this case.
\end{proof}

\subsection{Convergence to the mean field (Theorem~\ref{the:MeanFieldApproximation})}
\label{sec:ProofMeanFieldApproximation}

Recall that $\kappa=2$, and for simplicity, in this section we assume that $\mu=1$ and $\rho=1$.
Denote $f_t = \law(Z_t)$, where $Z_t$ is the unique solution of \eqref{eq:dZt}. Let $\cN(dt,du)$ and $\hat{\cM}(dt,du,dz)$ be independent Poisson random measures on $\R_+^2$ and $\R_+^3$ with respective intensities $dt\, du$ and $2 dt\, du\, f_t(dz)$, both independent of $Z_0$. Consider the SDE
\begin{equation}
	\label{eq:dZt2}
	dZ_t
	= -  Z_{t^\-} dt
	- Z_{t^\-} \int_0^\infty \ind{u\le \gamma Z_{t^\-}} \cN(dt,du)
	+ \int_0^\infty \int_0^\infty \ind{u \le \gamma z} \hat{\cM}(dt,du,dz).
\end{equation}

\begin{lemma}
	\label{the:well_posedness_Z2}
	Equation \eqref{eq:dZt2} admits a unique strong solution which has the same law as the solution of \eqref{eq:dZt}.
\end{lemma}

\begin{proof}
	\eqref{eq:dZt2} can be reduced to \eqref{eq:dZt} in a strong sense: define $\cM$ as the point measure on $\R_+^2$ with atoms $(t,u r_t / z)$ for every atom $(t,u,z)$ of $\hat{\cM}$ with $z>0$, where $r_t = \int_0^\infty z f_t(dz)$. It is straightforward to check that $\cM(dt,du)$ is a Poisson random measure with intensity measure $2 dt\, du$ and that \eqref{eq:dZt2} is exactly \eqref{eq:dZt}. The claim now follows from this reduction and Theorem~\ref{the:well_posedness_Z}.
\end{proof}

Let us now rewrite the dynamics \eqref{eq:dXti2ndVersion} as follows: Since the size of $K$ is $\kappa=2$, we note the poisson processes $(\cN_K)$ as $(\cN_{ij})_{1\leq i<j \leq N}$, where each $\cN_{ij}$ has intensity measure:
\[
\frac{2}{(N-1)(N-2)} dt \, du \, \ind{\lceil\xi\rceil\notin \{i,j\}}\ d\xi 
\]
and note $\cN_{ij}=\cN_{ji}$ whenever $i>j$. As stated before, observe that $\cN_{ij}$ provides the randomness for all possible jumps where neurons $i$ and $j$ are excited simultaneously, and the particle being reset corresponds to $\lceil\xi\rceil \in \{1,\ldots,N\} \backslash \{i,j\}$. With these building blocks, for $i\in\{1,\ldots,N\}$ we define the Poisson random measures
\begin{align*}
	\cN^i(dt,du) &= \sum_{\stackrel{k,j\neq i}{j<k},} \cN_{jk}(dt,du,(i-1,i])\\
	\cM^i(dt,du,d\xi)
	&= \sum_{j\neq i} \cN_{ij}(dt,du,d\xi).
\end{align*}
We note that $\cN^i$ encodes firing instants of neuron $i$, and $\cM^i$ encodes instants in which neuron $i$ is excited. Clearly, their respective intensities are 
\[
dt \, du
\qquad \text{and} \qquad
\frac{2\ \ind{\lceil\xi\rceil\neq i}}{N-1}\ dt \, du \,   d\xi   
\]

Thus, \eqref{eq:dXti2ndVersion} can be expressed as
\begin{equation}
	\label{eq:dXti2}
	\begin{split}
		dX_t^i
		&= - X_{t^\-}^i dt
		- X_{t^\-}^i \int_0^\infty \ind{u\le \gamma X_{t^\-}^i} \cN^i(dt,du) \\
		&\qquad \qquad \qquad
		+ \int_0^\infty \int_0^N \ind{u \le \gamma X_{t^\-}^{\lceil\xi\rceil}} \cM^i(dt,du,d\xi),
	\end{split}
\end{equation}
which closely resembles \eqref{eq:dZt2}. The only difference is that in the third term of \eqref{eq:dZt2} the variable $z$ is a sample from $f_t = \law(Z_t)$, while in \eqref{eq:dXti2} this is replaced by $X_{t^\-}^{\lceil\xi\rceil}$, which selects randomly a neuron $j \ne i$; that is, $X_{t^\-}^{\lceil\xi\rceil}$ is a sample from the (random) empirical measure $\bar{\mathbf{X}}_{t^\-}^i$. Here, we employ notations
\[
\bar{\mathbf{x}}
= \frac{1}{N} \sum_{i=1}^N \delta_{x^i}
\quad \text{and} \quad
\bar{\mathbf{x}}^i
= \frac{1}{N-1} \sum_{j\neq i} \delta_{x^j}.
\]

Motivated by this, for each $i$ we introduce a measurable mapping $(t,\mathbf{z},u) \mapsto F_t^i(\mathbf{z},u)$ from $\R_+ \times \R_+^N \times (0,N]$ into $\R_+$ such that
\[
\law( F_t^i(\mathbf{z}, U) )
\weq f_t
\]
and
\begin{equation}
	\label{eq:W1ftzi}
	\E \left[ \left\lvert F_t^i(\mathbf{z},U) - z^{\ceil{U}} \right\rvert \right] 
	\weq W_1(f_t, \bar{\mathbf{z}}^i)
\end{equation}
for all $t \in \R_+$, all $\mathbf{z} \in \R_+^N$, and for any random variable $U$ uniformly distributed on $(0,N] \setminus (i-1,i]$. By recalling that $\law( z^{\ceil{U}} ) = \bar{\mathbf{z}}^i$, we see that the pair $(F_t^i(\mathbf{z}, U), z^{\ceil{U}})$ constitutes an optimal coupling of $f_t = \law(Z_t)$ and the empirical distribution $ \bar{\mathbf{z}}^i = \frac{1}{N-1} \sum_{j \ne i} \delta_{z_j}$. See \cite[Lemma 3]{cortez2016quantitative} for a proof of existence of such a mapping.

Now, we specify our coupling by defining $\mathbf{Z}_t = (Z_t^1,\dots,Z_t^N)$ as a solution of
\begin{equation}
	\label{eq:dZti}
	\begin{aligned}
		dZ_t^i
		&\weq - Z_{t^\-}^i dt
		- Z_{t^\-}^i \int_0^\infty \ind{u\le \gamma Z_{t^\-}^i} \cN^i(dt,du) \\
		&\qquad + \int_0^\infty \int_{(0,N]} \ind{u \le \gamma F_t^i(\mathbf{Z}_{t^\-},\xi)} \cM^i(dt,du,d\xi),
	\end{aligned}
\end{equation}
started at $Z_0^i = X_0^i$ (recall that $X_0^1,\ldots,X_0^N$ are i.i.d.\ copies of $Z_0$). Here $\cN^i$ and $\cM^i$ are the same Poisson random measures appearing in \eqref{eq:dXti2}. Thanks to Lemma~\ref{the:well_posedness_Z2}, \eqref{eq:dZti} admits a unique strong solution, and it is a nonlinear process. Notice however that these processes are \textbf{not independent} because they have simultaneous jumps. However, their dependence vanishes as $N$ grows, according to the following result.

\begin{lemma}
	\label{lem:decouplingEW}
	There exists a constant $C>0$ such that
	\[
	\E\left[ W_1(f_t, \bar{\mathbf{Z}}_t)  \right]
	\wle \frac{C}{N^{1/3}}.
	\]
	The same is true with $\bar{\mathbf{Z}}_t^i$ in place of $\bar{\mathbf{Z}}_t$, for all $i\in\{1,\ldots,N\}$.
\end{lemma}

\begin{proof}[Proof of Lemma \ref{lem:decouplingEW}]
	
	By a direct application of \cite[Lemma 7]{cortez2016quantitative}, the following holds for every $k \le N$ and $t \ge 0$:
	\begin{equation}
		\label{lemma7Cortez}
		\begin{aligned} 
			\E[W_1(\bar{\mathbf{Z}}_t, f_t)]
			&\wle \frac{q k}{N}
			\left( W_1(f_t^{\otimes k}, \cL^k(\mathbf{Z}_t))
			+ \varepsilon_k(f_t) \right) \\
			&\quad + \frac{l}{N}\left( W_1(f_t^{\otimes l}, \cL^l(\mathbf{Z}_t))+ \varepsilon_l(f_t) \right)
		\end{aligned}
	\end{equation}
	with $N = k q+l$ and $\varepsilon_k(f_t) = \E \left[ W_1(\bar{\mbf{Y}}_t,f_t)\right]$ where $\mbf{Y}_t=(Y^1_t,...,Y^k_t)\sim f_t^{\otimes k}$ are mutually independent and $f_t$-distributed coordinates, and $\cL^k$ denotes the law of the first $k$ coordinates of an exchangeable random vector. Thus, it suffices to find suitable bounds for the quantities  $W_1(f_t^{\otimes k}, \cL^k(\mathbf{Z}_t))$ and $\varepsilon_k(f_t)$. We obtain such a bound for the first quantity with the aid of Lemma~\ref{lem:decouplingEW2}. This lemma states that there is a constant $C>0$ such that
	\[
	W_1(f_t^{\otimes k}, \cL^k(\mathbf{Z}_t))
	\wle C \frac{k}{N}
	\]
	for any $k\in\{2,...,N\}$.

	Let us now bound the other term. Making use of \cite[Thm. 1]{fournier2015rate} for $p=1, r\in(2,+\infty)$, we have:
	\[
	\varepsilon_k(f_t)
	\weq C(r) m_r^{1/r}(f_t)\left(k^{-1/2}+k^{\frac{r-1}{r}} \right)
	\wle  C'(r) m_r^{1/r}(f_t) k^{-1/2}
	\]
	Here $m_r(f_t) := \int z^r f_t(dz)$ and in the last inequality we used that $r>2$. Assume in \eqref{lemma7Cortez} that $k\sim N^{2/3}$. It is not difficult to verify that $l/N = (1+o(1)) N^{-1/3}$. Hence, if these conditions hold, it is sufficient to verify that
	\[
	W_1(f_t^{\otimes k}, \cL^k(\mathbf{Z}_t))
	+ \varepsilon_k(f_t)
	\wle \frac{C}{N^{1/3}} 
	\]
	If we assume that the $r$-th moment of $f_t$ is uniformly bounded, then
	\begin{equation}
		\label{lemmaLastIneq}
		W_1(f_t^{\otimes k}, \cL^k(\mathbf{Z}_t))
		+ \varepsilon_k(f_t)
		\wle C \frac{k}{N}+ C''(r) k^{-1/2}.
	\end{equation}
	Recall that $k = (1+o(1)) N^{2/3}$. Then the right hand side of \eqref{lemmaLastIneq} is essentially $C N^{-1/3} + C''(r) N^{-1/3}$, which is of order $N^{-1/3}$, so the result holds.
	
	Finally, we note that $m_1(f_t), m_2(f_t), m_3(f_t)$ are uniformly bounded by Lemma~\ref{the:UniformMomentBounds}.
\end{proof}

\begin{lemma}
	\label{lem:decouplingEW2}
	There exists a constant $C>0$ such that for every $k \in \{2,...,N\}$, and $t \ge 0$,
	\[
	W_1(f_t^{\otimes k}, \cL^k(\mathbf{Z}_t))
	\wle C \frac{k}{N}.
	\]
\end{lemma}

\begin{proof}[Proof of Lemma \ref{lem:decouplingEW2}]
	Fix an integer $2 \le k \le N$. The idea is to define a second coupling $\mbf{V}_t=(V^1_t,...,V^k_t)$ such that $\cL(\mbf{V}_t) = f_t^{\otimes k}$, and then prove that this new coupling is close in Wasserstein distance to $\cL^k(\mathbf{Z}_t)$.
	
	Recall that each neuron $i$ is associated to two random measures $\cN^i$ and $\cM^i$ on $\R_+^2$ and $\R_+^2\times(0,N]$ respectively, where 
	\begin{align*}
		\cN^i(dt,du) &= \sum_{\stackrel{k,j\neq i}{j<k},} \cN_{jk}(dt,du,(i-1,i]), \\
		\cM^i(dt,du,d\xi) &= \sum_{j \ne i} \cN_{ij}(dt,du,d\xi),
	\end{align*}
	with $(\cN_{ij})_{i<j}$ being independent processes in $\R_+^2\times(0,N]$ and $\cN_{ij} =\cN_{ji}$ if $j<i$.
	
	Furthermore, observe that in every interaction event, there are exactly 3 neurons involved (say $i<j<k$), and the corresponding processes for each potential event involving these neurons are $\cN_{ij}(dt,du,(k-1,k])$, $\cN_{jk}(dt,du,(i-1,i])$, and $\cN_{ik}(dt,du,(j-1,j])$. Hence, let $(\tilde{\cN}_{ij})_{i,j}$, $(\tilde{\tilde{\cN}}_{ij})_{i,j}$, and $(\tilde{\tilde{\tilde{\cN}}}_{ij})_{i,j}$ be independent copies of $(\cN_{ij})_{i,j}$. Define
	\begin{align*}
		R^i(dt,du,d\xi)
		&:=\sum_{j>k} \left[ \ind{\xi \in (k,N]}  \cN_{ij}(dt,du,d\xi)  + \ind{\xi \in (0,k]}  \tilde{\cN}_{ij}(dt,du,d\xi) \right]\\ 
		& + \sum_{i<j\le k} \left[ \ind{\xi \in (k,N]}  \tilde{\cN}_{ij}(dt,du,d\xi)  + \ind{\xi \in (0,k]}  \tilde{\tilde{\cN}}_{ij}(dt,du,d\xi) \right]  \\
		& + \sum_{j<i} \left[ \ind{\xi \in (k,N]}  \tilde{\tilde{\cN}}_{ij}(dt,du,d\xi)  + \ind{\xi \in (0,k]}  \tilde{\tilde{\tilde{\cN}}}_{ij}(dt,du,d\xi) \right],
	\end{align*}
	and let
	\[
	S^{i}(dt,du) := \cN^i(dt,du).
	\]
	
	Consequently, defining 
	\begin{align*}
		dV^i_t
		=& -V_{t^\-}^i dt-V^i_t \int_{\R_+} \ind{u\le \gamma V^i_{t^\-}} S^i(dt,du) \notag \\
		&+ \int_{\R_+} \int_{(0,N])}
		\ind{u \le \gamma F^i_{t^\-}(\mathbf{Z}_{t^\-},\xi)} R^i(dt,du,d\xi) 
	\end{align*}
	with $V^i_0 = \mathbf{Z}_0^i$, we get the desired coupling. Let us now verify that it remains close to $\cL^k(\mathbf{Z}_t)$. Observe that
	\[
	W_1(f_t^{\otimes k}, \cL^k(\mathbf{Z}_t))
	\wle \E\left[\frac{1}{k}\sum_{i=1}^{k}|Z^1_t-V_t^1|\right]
	\weq \E[|Z^1_t-V_t^1|]:=h_t,
	\]
	where we used interchangeability in the last equality. Then, it suffices to bound $h_t$. Let us denote $\cN{ij} := \cN{ij}(dr,du,d\xi)$. For $0 \le s \le t$:

	\begin{align}
		\label{eq:decouplingSDE}
		h_t &= h_s - \int_s^t h_r dr + \E  \int_s^t\int_0^{\infty}\Delta(Z^1_r,V_r^1)\   \cN^1(dr,du) \\
		&+ \E \int_s^t\int_0^{\infty}\int_0^k \ind{u\le \gamma F^1_{r^\-}(\mathbf{Z}_{r^\-},\xi)} \left(|Z^1_r+1-V_r^1|-|Z^1_r-V_r^1| \right) \sum_{j>k}\cN_{1j} \notag \\
		&+ \E \int_s^t\int_0^{\infty}\int_0^N \ind{u\le \gamma F^1_{r^\-}(\mathbf{Z}_{r^\-},\xi)} \left(|Z^1_r+1-V_r^1|-|Z^1_r-V_r^1| \right) \sum_{1<j\le k} \cN_{1j} \notag \\
		&+ \E \int_s^t\int_0^{\infty}\int_0^N \ind{u\le \gamma F^1_{r^\-}(\mathbf{Z}_{r^\-},\xi)} \left(|Z^1_r-V_r^1-1|-|Z^1_r-V_r^1| \right) \notag \\ 
		& \qquad \times \left[ \sum_{j>k} \ind{\xi \in (0,k]} \tilde{\cN}_{1j} +       \sum_{1<j\le k} \left( \ind{\xi \in (k,N]} \tilde{\cN}_{1j} + \ind{\xi \in (0,k]} \tilde{\tilde{\cN}}_{1j} \right)  \right] \notag \\
		& \le h_s - \int_s^t h_r dr + \E \int_s^t\int_0^{\infty}\int_0^N \ind{u\le \gamma F^i_{t^\-}(\mathbf{Z}_{t^\-},\xi)} \notag \\
		& \qquad \qquad \qquad \qquad \quad \times  \Bigg[ \sum_{j>k} \ind{\xi \in (0,k]} \left(\cN_{1j}+\tilde{\cN}_{1j}\right) +       \sum_{1<j\le k} \Big( \cN_{1j}\  + \notag \\
		& \qquad \qquad \qquad \qquad \qquad \quad \ind{\xi \in (k,N]} \tilde{\cN}_{1j} + \ind{\xi \in (0,k]} \tilde{\tilde{\cN}}_{1j} \Big)  \Bigg] \notag,
	\end{align}
	where $\E \int_s^t\int_0^{\infty}\Delta(Z^1_r,V_r^1)\   \cN^1(dr,du) = -\gamma \mathbb{E}\int_s^t |Z_r^1-V_r^1|^2\ dr$ which is negative.
	
	Now, observe that $\E\int_{i}^{i+1}|F^i_{t^\-}(\mathbf{Z}_{t^\-},\xi)| d\xi = m_1(f_t)$, and denoting $I(\cN)$ the intensity of a Poisson process $\cN$, observe that, for each Poisson process in \eqref{eq:decouplingSDE}: 
	
	\begin{itemize}
		\item $ I\left( \ind{\xi \in (0,k]} \sum_{j>k}\left(\cN_{1j}+\tilde{\cN}_{1j}\right) \right) = 2. \frac{2(N-k)}{(N-1)(N-2)} \ind{\xi \in (1,k]}\ \ dr\ du\ d\xi$
		\item $I\left( \sum_{1<j\le k} \cN_{1j} \right)= \frac{2}{(N-1)(N-2)}\left[(k-2)\ind{\xi \in (1,k]}+ (k-1)\ind{\xi \in (k,N]} \right] \ \ dr\ du\ d\xi $
		\item $I\left( \sum_{1<j\le k} \ind{\xi \in (k,N]} \tilde{\cN}_{1j} \right)= \frac{2(k-1)}{(N-1)(N-2)} \ind{\xi \in (k,N]}\ \ dr\ du\ d\xi $
		\item $I\left( \sum_{1<j\le k} \ind{\xi \in (0,k]} \tilde{\tilde{\cN}}_{1j}  \right) = \frac{2(k-2)}{(N-1)(N-2)} \ind{\xi \in (1,k]}\ \ dr\ du\ d\xi $
	\end{itemize}
	
	Then, the following inequality holds:
	\begin{align*}
		h_t - h_s
		&\wle - \int_s^t h_r dr + \frac{4(N-k)(k-1)\gamma}{(N-1)(N-2)}\int_s^t m_1(f_r) dr \\
		&\quad + \frac{2(k-2)(k-1)\gamma}{(N-1)(N-2)} \int_s^t m_1(f_r) dr \\
		&\quad + 2 \frac{2(k-2)(N-k)\gamma}{(N-1)(N-2)} \int_s^t m_1(f_r) dr \\
		&\quad + \frac{2(k-2)(k-1)\gamma}{(N-1)(N-2)}\int_s^t m_1(f_r) dr \\
		&\weq -\int_s^th_r dr
		+ \gamma \frac{4[(N-k)(2k-3)+(k-2)(k-1)]}{(N-1)(N-2)} \int_s^t m_1(f_r) dr.
	\end{align*}
	Considering that $(2k-3)\le 2(k-1)$ and $(k-1)(k-2)\le 2(k-1)(k-2)$, we have
	\begin{align*}
		h_t - h_s
		&\wle -\int_s^th_r dr
		+ 8\gamma \frac{(k-1) }{(N-1)} \int_s^t m_1(f_r) dr.
	\end{align*}
	Hence,
	\begin{equation*}
		\partial_t h_t
		\wle - h_t + 8\gamma \frac{(k-1) }{(N-1)} m_1(f_t)
	\end{equation*}
	Because $m_1(f_t) \le M$ is uniformly bounded, Gronwall's inequality now implies that
	\begin{equation*}
		h_t
		\wle 8 M \gamma\frac{(k-1)}{(N-1)}
		\left( 1 - e^{-t} \right),
	\end{equation*}
	and this gives us the desired result.
\end{proof}

\begin{lemma}
	\label{the:UniformMomentBounds}
	Denote $m_r(f_t) = \int z^r f_t(dz) = \E Z_t^r$ for $f_t = \law(Z_t)$. Fix an integer $1 \le r \le 3$. If $m_r(f_0) < \infty$, then $\sup_{t \ge 0} m_r(f_t) < \infty$.
\end{lemma}
\begin{proof}
	Observe that $\E \Big((Z_t+\rho)^r - Z_t^r \Big) = \sum_{q=0}^{r-1} \binom{r}{q} \rho^{r-q} \E Z_t^q$.
	Dynkin's formula applied to \eqref{eq:dZt} implies that
	\begin{equation}
		\label{eq:DynkinMoment}
		\frac{d}{dt} m_r(f_t)
		\weq - \mu r m_r(f_t) - \gamma m_{r+1}(f_t)
		+ \kappa \gamma m_1(f_t) \sum_{q=0}^{r-1} \binom{r}{q} \rho^{r-q} m_q(f_t).
	\end{equation}
	For $r=1$, \eqref{eq:DynkinMoment} becomes
	$
	\frac{d}{dt} m_1(f_t)
	= (\kappa\rho\gamma-\mu) m_1(f_t) - \gamma m_2(f_t).
	$
	Therefore, $m_2(f_t) \le m_1(f_t)^2$ implies that
	\[
	\frac{d}{dt} m_1(f_t)
	\wle (\kappa\rho\gamma-\mu) m_1(f_t)
	- \gamma m_1(f_t)^2.
	\]
	This inequality shows that $\frac{d}{dt} m_1(f_t) \le 0$ whenever $m_1(f_t) \ge \kappa\rho-\mu/\gamma$. Therefore,
	\[
	c_1
	\weq \sup_{t \ge 0} m_1(f_t)
	\wle \max \{ m_1(f_0), \ \kappa\rho-\mu/\gamma\}.
	\]
	
	Next, \eqref{eq:DynkinMoment} for $r=2$ implies that
	\begin{align*}
		\frac{d}{dt} m_2(f_t)
		&\weq - 2\mu m_2(f_t) - \gamma m_3(f_t)
		+ \kappa \gamma m_1(f_t) \Big( \rho^2 + 2\rho m_1(f_t) \Big) \\
		&\wle - 2\mu m_2(f_t)
		+ \kappa \gamma (c_1+\rho)^3,
	\end{align*}
	which shows that $\frac{d}{dt} m_2(f_t) \le 0$ whenever $m_2(t) \ge (2\mu)^{-1} \kappa \gamma (c_1+\rho)^3$. Therefore,
	\[
	c_2
	\weq \sup_{t \ge 0} m_2(f_t)
	\wle \max \{ m_2(f_0), \ (2\mu)^{-1} \kappa \gamma (c_1+\rho)^3 \}.
	\]
	Similarly, \eqref{eq:DynkinMoment} for $r=3$ implies that
	\begin{align*}
		\frac{d}{dt} m_3(f_t)
		&\weq - 3\mu m_3(f_t) - \gamma m_4(f_t)
		+ \kappa\gamma m_1(f_t) \Big( \rho^3 + 3\rho^2 m_1(f_t) + 3 \rho m_2(f_t) \Big) \\
		&\wle - 3\mu m_3(f_t)
		+ \kappa\gamma (c_1 \vee c_2 + \rho)^4.
	\end{align*}
	The same argument as for $r=2$ now shows that
	\[
	c_3
	\weq \sup_{t \ge 0} m_3(f_t)
	\wle \max \{ m_3(f_0), \ (3\mu)^{-1} \kappa \gamma (c_1 \vee c_2 +\rho)^4 \}.
	\qedhere
	\]
\end{proof}

\begin{proof}[Proof of Theorem \ref{the:MeanFieldApproximation}]
	Because $\frac{1}{N} \sum_{i=1}^N \delta_{(Z^i_t, X^i_t)}$ constitutes a coupling of $\bar{\mathbf{Z}}_t$ and $\bar{\mathbf{X}}_t$, we see that 
	\begin{align*}
		W_1(f_t, \bar{\mathbf{X}}_t)
		\wle W_1(f_t, \bar{\mathbf{Z}}_t)
		+ W_1(\bar{\mathbf{Z}}_t, \bar{\mathbf{X}}_t)
		\wle W_1(f_t, \bar{\mathbf{Z}}_t)
		+ \frac{1}{N} \sum_{i=1}^N |Z_t^i - X_t^i|,
	\end{align*}
	so due to exchangeability, we find that
	\begin{align*}
		\E \left[W_1(f_t, \bar{\mathbf{X}}_t)\right]
		&\wle \E \left[W_1(f_t, \bar{\mathbf{Z}}_t)\right] + h_t,
	\end{align*}
	where $h_t := \E\left[|Z_t^1 - X_t^1|\right]$. By Lemma \ref{lem:decouplingEW}, it suffices to estimate $h_t$.
	
	From \eqref{eq:dXti2} and \eqref{eq:dZti} we see that
	\[
	\partial_t h_t
	\wle - h_t - \gamma \E \left[Z_t^1 - X_t^1)^2\right]
	+ 2 \gamma \E \int_1^N \left|F_t^1(\mathbf{Z}_t, \xi) - X_t^{\lceil\xi\rceil}\right| \frac{d\xi}{N-1}.
	\]
	We just discard the second term, while in the third term we add and subtract $Z_t^{\lceil\xi\rceil}$. Using \eqref{eq:W1ftzi} and exchangeability, we thus obtain
	\begin{align*}
		\partial_t h_t
		&\wle -h_t + 2\gamma \E\left[W_1(f_t,\bar{\mathbf{Z}}_t^1)\right] + 2 \gamma \E\left[ \frac{1}{N-1} \sum_{i=2}^N |Z_t^i - X_t^i| \right] \\
		&\wle (\theta_c-1) h_t + \frac{C}{N^{1/3}},
	\end{align*}
	thanks to Lemma~\ref{lem:decouplingEW}. The claim follows using Gronwall's lemma.
\end{proof}

\subsection{Persistence (Theorem \ref{the:persist})}

Recall that $\kappa=2$. We also assume that $\rho=\mu=1$.

\begin{proof}[Proof of Theorem \ref{the:persist}]
	We first treat the case $\theta>1$, which implies that $\theta_c>1$. Call $\epsilon = \frac{1}{2}\inf_{t\geq 0} \E[Z_t]$; thanks to Theorem \ref{the:kappa_alpha}, we know that $\epsilon >0$. Then, for any $t\geq 0$, we have
	\[
	2\epsilon
	\leq \E[Z_t]
	= W_1(f_t, \delta_0)
	\leq \E[ W_1(f_t, \bar{\mathbf{X}}_t)] + \E[ W_1(\bar{\mathbf{X}}_t, \delta_0)]
	\leq \frac{Ce^{(\theta_c-1) t}}{N^{1/3}} + \E[X_t^1],
	\]
	where in the last step we used Theorem \ref{the:MeanFieldApproximation} and exchangeability. Clearly, $Ce^{(\theta_c-1) t} N^{-1/3} \leq \epsilon$ if and only if $t \leq \tilde{c} + c \log N$ for $c = 1/[3(\theta_c-1)] >0$ and $\tilde{c} = [1/(\theta_c-1)] \log(\epsilon/C) \in \R$, from which the result follows.
	
	We now treat the case $\theta<1$. The convergence in the metric $\omega(\cdot,\cdot)$ is exactly \eqref{eq:wellpsdFinite}, so there's nothing to prove. In the case $2\gamma = \theta_c < 1$, from \eqref{eq:dXti2} it is easy to see that
	\[
	\frac{d}{dt} \E[X_t^1]
	= - \E[X_t^1] - \gamma \E[(X_t^1)^2] + \frac{2\gamma}{N-1} \sum_{i=2}^N \E[X_t^i]
	\leq -(1-\theta_c) \E[X_t^1],
	\]
	where in the last step we simply discarded the term $-\gamma \E[(X_t^1)^2] \leq 0$, and we used exchangeability. The result now follows from Grönwall's lemma.
\end{proof}

\section{Possible further research}
\label{sec:future_work}

Several lines of work can be derived from these results. For instance, we are confident that the process $Z_t$ has a unique nontrivial limit $Z_{\infty}$
in the case $\theta >1$, see Remark. This is supported by numerical simulations of the finite network for relatively large $N$: the empirical distribution of the process seems to stabilize when one lets the simulation run for a long time (even considering the fact that Theorem \ref{the:wellpsdFinite} implies that $\mbf{X}_t$ must eventually decay, see Remark \ref{remark:logN}). It would be very desirable to have a rigorous proof of this phenomenon.

An interesting extension would be to consider the firing range $\kappa$ and/or the magnitude of the excitatory impulse $\rho$ to be random. We believe that such a model should lead essentially to the same results (replacing $\kappa$ and $\rho$ with $\E[\kappa]$ and $\E[\rho]$, respectively) for sufficiently well behaved $\kappa$ and $\rho$; see Remark \ref{remark:general_kappa}.

Finally, a natural generalization would be the introduction of a spatial structure on the model, for instance by considering a random graph representing the interconnections between neurons. This can possibly lead to interesting results, depending on the underlying structure of the graph.

\bibliographystyle{abbrv}
\bibliography{persistence}

\begin{thebibliography}{10}

\bibitem{baladron2012mean}
J.~Baladron, D.~Fasoli, O.~Faugeras, and J.~Touboul.
\newblock Mean-field description and propagation of chaos in networks of
  {Hodgkin-Huxley} and {FitzHugh-Nagumo} neurons.
\newblock {\em Journal of Mathematical Neuroscience}, 2(1):10, 2012.

\bibitem{BENACHOUR1998173}
S.~Benachour, B.~Roynette, D.~Talay, and P.~Vallois.
\newblock Nonlinear self-stabilizing processes – i existence, invariant
  probability, propagation of chaos.
\newblock {\em Stochastic Processes and their Applications}, 75(2):173 -- 201,
  1998.

\bibitem{villani}
F.~Bolley, A.~Guillin, and C.~Villani.
\newblock Quantitative concentration inequalities for empirical measures on
  non-compact spaces.
\newblock {\em Probability Theory and Related Fields}, 137(3):541--593, 2007.

\bibitem{burkitt2006review}
A.~N. Burkitt.
\newblock A review of the integrate-and-fire neuron model: I. homogeneous
  synaptic input.
\newblock {\em Biological Cybernetics}, 95(1):1--19, 2006.

\bibitem{burkitt2006review2}
A.~N. Burkitt.
\newblock A review of the integrate-and-fire neuron model: Ii. inhomogeneous
  synaptic input and network properties.
\newblock {\em Biological Cybernetics}, 95(2):97--112, 2006.

\bibitem{caceres2011analysis}
M.~J. C{\'a}ceres, J.~A. Carrillo, and B.~Perthame.
\newblock Analysis of nonlinear noisy integrate \& fire neuron models: blow-up
  and steady states.
\newblock {\em Journal of Mathematical Neuroscience}, 1(1):7, 2011.

\bibitem{carrillo2020long}
J.~Carrillo, R.~Gvalani, G.~Pavliotis, and A.~Schlichting.
\newblock Long-time behaviour and phase transitions for the {McKean--Vlasov}
  equation on the torus.
\newblock {\em Archive for Rational Mechanics and Analysis}, 235(1):635--690,
  2020.

\bibitem{chichilnisky2001simple}
E.~Chichilnisky.
\newblock A simple white noise analysis of neuronal light responses.
\newblock {\em Network: Computation in Neural Systems}, 12(2):199--213, 2001.

\bibitem{cortez2016quantitative}
R.~Cortez and J.~Fontbona.
\newblock Quantitative propagation of chaos for generalized {Kac} particle
  systems.
\newblock {\em Annals of Applied Probability}, 26(2):892--916, 2016.

\bibitem{Davis_1984}
M.~H.~A. Davis.
\newblock Piecewise-deterministic {Markov} processes: {A} general class of
  non-diffusion stochastic models.
\newblock {\em Journal of the Royal Statistical Society B}, 46(3):353--388,
  1984.

\bibitem{delarue2015global}
F.~Delarue, J.~Inglis, S.~Rubenthaler, and E.~Tanr{\'e}.
\newblock Global solvability of a networked integrate-and-fire model of
  {McKean--Vlasov} type.
\newblock {\em Annals of Applied Probability}, 25(4):2096--2133, 2015.

\bibitem{fournier2015rate}
N.~Fournier and A.~Guillin.
\newblock On the rate of convergence in {Wasserstein} distance of the empirical
  measure.
\newblock {\em Probability Theory and Related Fields}, 162(3-4):707--738, 2015.

\bibitem{galloway-woo-lu2008}
E.~M. Galloway, N.~H. Woo, and B.~Lu.
\newblock Persistent neural activity in the prefrontal cortex: {A} mechanism by
  which {BDNF} regulates working memory?
\newblock In W.~S. Sossin, J.-C. Lacaille, V.~F. Castellucci, and
  S.~Belleville, editors, {\em Essence of Memory}, volume 169 of {\em Progress
  in Brain Research}, pages 251--266. Elsevier, 2008.

\bibitem{hodgkin1952quantitative}
A.~L. Hodgkin and A.~F. Huxley.
\newblock A quantitative description of membrane current and its application to
  conduction and excitation in nerve.
\newblock {\em Journal of Physiology}, 117(4):500, 1952.

\bibitem{kac1956}
M.~Kac.
\newblock Foundations of kinetic theory.
\newblock In {\em Proceedings of the {T}hird {B}erkeley {S}ymposium on
  {M}athematical {S}tatistics and {P}robability, 1954--1955, vol. {III}}, pages
  171--197, Berkeley and Los Angeles, 1956. University of California Press.

\bibitem{lapicque1907recherches}
L.~Lapicque.
\newblock Recherches quantitatives sur l'excitation electrique des nerfs
  traitee comme une polarization.
\newblock {\em Journal de Physiologie et de Pathologie Generalej}, 9:620--635,
  1907.

\bibitem{major-tank2004}
G.~Major and D.~Tank.
\newblock Persistent neural activity: {P}revalence and mechanisms.
\newblock {\em Current Opinion in Neurobiology}, 14(6):675--684, 2004.

\bibitem{malrieu}
F.~Malrieu.
\newblock Logarithmic {Sobolev} inequalities for some nonlinear {PDE}'s.
\newblock {\em Stochastic Processes and their Applications}, 95(1):109 -- 132,
  2001.

\bibitem{mischler-mouhot2013}
S.~Mischler and C.~Mouhot.
\newblock Kac's program in kinetic theory.
\newblock {\em Invent. Math.}, 193(1):1--147, 2013.

\bibitem{pillow2005prediction}
J.~W. Pillow, L.~Paninski, V.~J. Uzzell, E.~P. Simoncelli, and E.~Chichilnisky.
\newblock Prediction and decoding of retinal ganglion cell responses with a
  probabilistic spiking model.
\newblock {\em Journal of Neuroscience}, 25(47):11003--11013, 2005.

\bibitem{pillow2008spatio}
J.~W. Pillow, J.~Shlens, L.~Paninski, A.~Sher, A.~M. Litke, E.~Chichilnisky,
  and E.~P. Simoncelli.
\newblock Spatio-temporal correlations and visual signalling in a complete
  neuronal population.
\newblock {\em Nature}, 454(7207):995--999, 2008.

\bibitem{robert2016dynamics}
P.~Robert and J.~Touboul.
\newblock On the dynamics of random neuronal networks.
\newblock {\em Journal of Statistical Physics}, 165(3):545--584, 2016.

\bibitem{rolls2010noisy}
E.~T. Rolls and G.~Deco.
\newblock {\em The noisy brain: {S}tochastic dynamics as a principle of brain
  function}.
\newblock Oxford University Press, 2010.

\bibitem{sznitman1989}
A.-S. Sznitman.
\newblock Topics in propagation of chaos.
\newblock In {\em {\'E}cole d'{\'E}t{\'e} de {P}robabilit{\'e}s de
  {S}aint-{F}lour {XIX}---1989}, volume 1464 of {\em Lecture Notes in Math.},
  pages 165--251. Springer, Berlin, 1991.

\bibitem{touboul2012mean}
J.~Touboul.
\newblock Mean-field equations for stochastic firing-rate neural fields with
  delays: Derivation and noise-induced transitions.
\newblock {\em Physica D: Nonlinear Phenomena}, 241(15):1223--1244, 2012.

\bibitem{touboul2014propagation}
J.~Touboul.
\newblock Propagation of chaos in neural fields.
\newblock {\em Annals of Applied Probability}, 24(3):1298--1328, 2014.

\bibitem{touboul2014spatially}
J.~Touboul.
\newblock Spatially extended networks with singular multi-scale connectivity
  patterns.
\newblock {\em Journal of Statistical Physics}, 156(3):546--573, 2014.

\bibitem{tugaut}
J.~Tugaut.
\newblock Phase transitions of {McKean--Vlasov} processes in double-wells
  landscape.
\newblock {\em Stochastics}, 86(2):257--284, 2014.

\bibitem{zylberberg-strowbridge2017}
J.~Zylberberg and B.~W. Strowbridge.
\newblock Mechanisms of persistent activity in cortical circuits: Possible
  neural substrates for working memory.
\newblock {\em Annual Review of Neuroscience}, 40(1):603--627, 2017.

\end{thebibliography}

\end{document}